\documentclass[12pt]{amsart}

\usepackage{amsmath}
\usepackage{amsfonts}
\usepackage{amssymb}
\usepackage{amsthm} 
\usepackage[english]{babel}
\usepackage[latin1]{inputenc}
\usepackage[T1]{fontenc}
\usepackage{graphicx}
\usepackage{enumitem}

\usepackage{geometry}
\usepackage{hyperref}

\DeclareMathOperator{\Diff}{Diff}
\DeclareMathOperator{\Ric}{Ric}
\DeclareMathOperator{\interior}{Int}
\DeclareMathOperator{\spt}{spt}
\DeclareMathOperator{\Id}{Id}

\DeclareMathOperator{\dmn}{dmn}
\DeclareMathOperator{\Vol}{Vol}

\newtheorem{theo}{Theorem}[]

\newtheorem{lemme}[theo]{Lemma}
\newtheorem{definition}[theo]{Definition}
\newtheorem{coro}[theo]{Corollary}
\newtheorem{conj}[theo]{Conjecture}
\newtheorem{remarque}[theo]{Remark}

\begin{document}

\title[Local min-max surfaces and minimal Heegaard splittings]{Local min-max surfaces and strongly irreducible minimal Heegaard splittings} 
\author{Antoine Song} \thanks{The author was supported by NSF-DMS-1509027.}

\begin{abstract}
Let $(M,g)$ be a closed oriented Riemannian $3$-manifold not diffeomorphic to the 3-sphere and suppose that there is a strongly irreducible Heegaard splitting $H$. We prove that $H$ is either isotopic to a minimal surface of index at most one or isotopic to the stable oriented double cover of a non-orientable minimal surface with a vertical handle attached. In particular, this proves a result conjectured by H. J. Rubinstein. Some consequences include the existence in any $\mathbb{R}P^3$ of either a minimal torus or a minimal projective plane with stable universal cover. In the case of positive scalar curvature, we show for spherical space forms not diffeomorphic to $S^3$ or $\mathbb{R}P^3$ that any strongly irreducible Heegaard splitting is isotopic to a minimal surface, and that there is a minimal Heegaard splitting of area less than $4\pi$ if $R\geq 6$.

\end{abstract}

\maketitle

\section*{}

When studying the topology of $3$-manifolds, it can be useful to realize a Heegaard splitting as a minimal surface, since the geometric nature of minimal surfaces simplifies comparisons between splittings and helps in classification or counting problems. Such a result was announced by Rubinstein in \cite[Theorem 1.8]{Rubinsteinnotes}, where he sketches a proof, and it was used for instance by Colding, Gabai \cite{ColdingGabai} and Colding, Gabai, Ketover \cite{ColdingGabaiKetover} to prove an effective version of the classification of irreducible Heegaard splittings in non-Haken hyperbolic $3$-manifolds first proved by Tao Li with other methods \cite{TaoLI1} \cite{TaoLI2} \cite{TaoLI3}. Apparently a complete proof of \cite[Theorem 1.8]{Rubinsteinnotes} was not published, so we state it as a conjecture:

\begin{conj}[Rubinstein \cite{Rubinsteinnotes}] 

Let $(M,g)$ be a closed oriented Riemannian $3$-manifold not diffeomorphic to the $3$-sphere and suppose that there is a strongly irreducible Heegaard splitting $H$. Then $H$ is either isotopic to a minimal surface or isotopic to the oriented double cover of a non-orientable minimal surface with a vertical handle attached.
\end{conj}

Since a Heegaard splitting determines a continuous family of surfaces foliating the manifold and whose two ends are graphs (the spines of the two handlebodies), a natural direction to prove such a result is the min-max theory for minimal surfaces. Its principle can be roughly described as follows: taking a sequence of $1$-parameter smooth families of surfaces $\{\Sigma^i_t\}_{t\in[0,1]}$ sweeping out the manifold, which are tighter and tighter in the sense that $\max_t \mathcal{H}^2(\Sigma^i_t)$ converges to the infimum possible among sweepouts in a same homotopy class, there exists a subsequence of surfaces $\Sigma^{j}_{t_{j}}$ converging to a minimal surface.

Based on this observation, Rubinstein outlined a proof that in the case when the Heegaard splitting $H$ is strongly irreducible then the min-max theory should create a minimal surface isotopic to $H$ or whose double cover plus a vertical neck is isotopic to $H$. Since the minimal surface obtained by min-max could be broken into several components and thus not isotopic to $H$, his idea was to iterate the min-max procedure until the desired situation occurred. The proof sketched by Rubinstein, while insightful and natural, remains vague or incomplete on essentially two points. Based on another unpublished announced result of Pitts and Rubinstein \cite{PittsRubinstein}, he assumes that the min-max surface essentially contains the topological information of the sweepouts, and this was proved later by Ketover \cite{Ketgenusbound}. Moreover he suggests to apply the min-max theorem but taking into account only a subdomain of the manifold, and he claims that one gets a minimal surface inside the interior of such a subdomain. This needs to be defined and justified (see Theorem \ref{smoothminmax}). Besides, in \cite[Corollary 1.9]{Rubinsteinnotes} another minor point is that he takes the limit of min-max surfaces by stating that their areas are bounded, which is unclear since the min-max widths he just considered are attached to subdomains and could a priori be unbounded, even if the metrics converge.

In \cite{Ketgenusbound}, Ketover analyzed how the subsequence $\Sigma^{j}_{t_{j}}$ converges to a min-max limit surface $\Sigma^\infty$. He deduced a genus bound for $\Sigma^\infty$ which was conjectured by Pitts and Rubinstein. He also made an important step towards a proof of Rubinstein's claimed result, by proving that either it was true, or $\Sigma^\infty$ had several components and was obtained by some surgeries. The problem left to get the final result was that $\Sigma^\infty$ may be broken into several components. 

In this paper, we prove the conjecture thanks to a local version of the min-max theorem. In fact,   the issue appears only in the case the min-max limit has more than one non-sphere component and since they bound handlebodies under our assumptions, we want to get rid of such possible handlebodies and 3-balls with stable minimal boundaries of $M$. Thus we reformulate the strategy of \cite{Rubinsteinnotes} as follows: instead of iterating the min-max procedure we remove a maximal disjoint union of undesired handlebodies that could appear after min-max, we get a compact $3$-manifold $N$ with boundary called core (see Subsection \ref{corecore}), and we apply the local min-max to $N$. The min-max surface, which we prove is not included in the boundary of $N$, cannot be non-trivially broken into several components by construction of $N$ and the conjecture is proved. We insist that the crucial step (and main result of this paper) is to show that we obtain a minimal surface that is not entirely included in the boundary $\partial N$. 

\begin{theo} 
Let $(M,g)$ be a closed oriented $3$-manifold. Suppose that there is a strongly irreducible Heegaard splitting $H$ of $M$ which, in the case where $M$ is a $3$-sphere, is supposed to be a $2$-sphere. Then either $H$ is isotopic to a minimal surface of index at most one, or isotopic to the stable minimal oriented double cover of a non-orientable minimal surface $\Sigma_2$ with a vertical handle attached.

\end{theo}

This statement is actually slightly more precise than the original conjecture because if $H$ is not isotopic to a minimal surface of index at most one, then it is isotopic to the \textit{stable} oriented double cover of a non-orientable surface with a vertical handle attached. For this precise improvement, the catenoid estimate of Ketover, Marques and Neves \cite{KeMaNe} will be useful. Notice that the proof using min-max does not work for the torus splitting of the $3$-sphere: in sufficiently round spheres, starting from a strongly Heegaard splitting of genus one, min-max can only give minimal spheres and not tori. The technical reason is that such a torus in $S^3$ is stabilized, even though it is irreducible.

We note that in \cite{Montezuma1} Montezuma proved another kind of localized min-max theorem for a domain $\Omega$, but it differs from our version in the following way. In \cite{Montezuma1}, $\Omega$ is mean-concave, the surfaces sweep out the entire manifold and only the ones touching $\Omega$ are considered when defining the partial min-max width. In our version $\Omega$ has minimal or mean-convex boundary, and the surfaces only sweep out the domain $\Omega$. Moreover, our version uses the Simon-Smith setting so the analysis of \cite{Ketgenusbound} can be applied, whereas in \cite{Montezuma1}, Theorem B is proved in the Almgren-Pitts setting and no topological information can be deduced for the min-max minimal surface (though a smooth version of it is expected).


As a corollary of the proof of Rubinstein's conjecture, we obtain:

\begin{coro} \label{feinherb}
\begin{enumerate}
\item Any lens space not diffeomorphic to $S^3$ or $\mathbb{R}P^3$ contains a minimal torus with index at most one.
\item Any $\mathbb{R}P^3$ contains either a minimal torus of index at most one or a minimal projective plane with stable universal cover. In particular if the metric is bumpy, then either there is an index one minimal torus or there is an index one minimal sphere.
\end{enumerate}
 \end{coro}

Ketover was also aware of these applications of the conjecture and the second part of the second item was observed to me by Marques. The mapping degree method of White \cite{Whitebumpy} gives the existence of a minimal torus in every $\mathbb{R}P^3$ with positive Ricci curvature. This result was improved in \cite[Theorem 3.3]{KeMaNe} to obtain an index one minimal torus in such a manifold. The previous corollary can be thought of as an extension of these results in the direction of finding index at most one minimal tori. We will give an example showing that the second item in Corollary \ref{feinherb} is optimal for index at most one minimal tori not included in a $3$-ball. 

Specializing these results to the case of positive scalar curvature $R$, we first prove:

\begin{theo}
In a spherical space form not diffeomorphic to $S^3$ or $\mathbb{R}P^3$ with positive scalar curvature, any strongly irreducible Heegaard splitting admits an index one minimal representative in its isotopy class.
\end{theo}

When $R\geq 6$, the Hersch trick for index one oriented minimal surfaces only gives $16\pi/3$ as an upper area bound when the genus is odd. However when $\Ric>0$, Marques and Neves proved in \cite{MaNe} that in many cases there is a Heegaard splitting of index one and area less than $4\pi$. By extending our method in \cite{Antoine}, where it is proved that there always exists a minimal surface of area at most $4\pi$ when $R\geq 6$, we note the following generalization:

 \begin{theo}
Any spherical space form not diffeomorphic to $S^3$ or $\mathbb{R}P^3$ with $R\geq 6$ admits an index one minimal Heegaard splitting of area less than $4\pi$. In an $\mathbb{R}P^3$ with $R\geq6$, either there is an index one minimal Heegaard splitting of genus $1$ and with area less than $4\pi$ or there is a minimal $\mathbb{R}P^2$ with stable universal cover and with area less than $2\pi$.
\end{theo}
 
The proof uses the local min-max theorem applied to the lift of the manifold to $S^3$. Notice that in the two previous theorems, the first one gives a minimal genus Heegaard splitting while the second one gives an improved area bound for a (not necessarily irreducible) Heegaard splitting. It is tempting to can ask whether these two results can be combined (see Remark \ref{physalis}). It can also be interesting to compare these results with \cite{Montezuma2}, where Montezuma constructs metrics with positive scalar curvature and unbounded min-max widths.

Since all the min-max constructions are localized, versions of the previous results hold for non-prime $3$-manifolds, and manifolds with boundary.

After the first version of this paper was available, a work of D. Ketover and Y. Liokumovich \cite{KetLio} on related questions and techniques appeared.

\subsection*{Acknowledgement} 
I am grateful to my advisor Fernando Cod{\'a} Marques for his support, for engaging discussions about the variational theory of minimal surfaces, and for stimulating questions and answers. I thank Rafael Montezuma for several clarifying conversations about his articles, Amitesh Datta and Maggie Miller for discussions about the topology of $3$-manifolds and Davi Maximo for his interest.


\section{Topological preliminaries and some notations} \label{basics}

\subsection{} \textbf{Generalized Heegaard splittings}

We begin some basic definitions and notations in topology.

A 3-manifold $W$ is a \textit{compression body} if there is a connected closed oriented surface $S$ such that $W$ is obtained from $S\times[0,1]$ by attaching $2$-handles along mutually disjoint loops in $S\times\{1\}$ and filling in some resulting $2$-sphere boundary components with $3$-handles. We do not require to fill in all the $2$-sphere boundary components. Denote $S\times\{0\}$ by $\partial_+W$ and $\partial W \backslash \partial_+W$ by $\partial_-W$. A compression body is called a \textit{handlebody} if $\partial_-W=\varnothing$. It is said to be \textit{trivial} if $W\approx\partial_+ W\times [0,1]$.

Let $M$ be a connected compact oriented $3$-manifold. When $M$ is closed, an embedded connected orientable surface $H$ is a \textit{Heegaard splitting} if $M\backslash H$ has two connected components each diffeomorphic to a handlebody. This notion can be generalized to the case when $\partial M\neq \varnothing$ as follows. Let $(\partial_0M,\partial_1 M)$ be a partition of the boundary components of $M$. A triplet $(W_0,W_1,H)$ is called a \textit{generalized Heegaard splitting} of $(M,\partial_0M,\partial_1M)$ if $W_0$, $W_1$ are compression bodies with 
$$W_0\cup W_1=M, \quad \partial_-W_0=\partial_0M, \quad \partial_-W_1 = \partial_1M$$
 $$\text{ and } W_0\cap W_1 = \partial_+W_0=\partial_+W_1=H.$$ The surface $H$ is then also called a \textit{generalized Heegaard splitting} of $M$. The \textit{Heegaard genus} of $M$ is the lowest possible genus of a generalized Heegaard splitting of $M$. Recall that any triplet $(\partial_0M,\partial_1 M)$ as above admits a generalized Heegaard splitting. 

Let $S$ be an orientable closed surface embedded in the interior of $M$ as above. An \textit{essential disk} in $(M,S)$ is a disk $D$ embedded in $M$ such that $D\cap S=\partial D$ and $\partial D$ is an essential curve in $S$. 

A generalized Heegaard splitting $(W_0,W_1,H)$ is called \textit{irreducible} when there are no essential disks $(D_0,\partial D_0) \subset (W_0,H)$, $(D_1,\partial D_1) \subset (W_1,H)$ such that $\partial D_0 = \partial D_1$. We say that $(W_0,W_1,H)$ is \textit{strongly irreducible} if 
there are no essential disks $(D_0,\partial D_0) \subset (W_0,H)$, $(D_1,\partial D_1) \subset (W_1,H)$ such that $\partial D_0 \cap \partial D_1 = \varnothing$.


\subsection{}\label{corecore}\textbf{H-compatible splittings and H-cores}

Let $(M,g)$ be a connected compact oriented $3$-manifold whose boundary, if non-empty, is a stable minimal surface. In this subsection, we study generalized Heegaard splittings of some compact $3$-manifolds $N$ included in $(M,g)$, and whose boundary is also a stable minimal surface. 

In this subsection, we will consider $(W_0,W_1,H)$ a generalized Heegaard splitting of $(M,\partial_0M,\partial_1M)$ where $(\partial_0M,\partial_1M)$ is a given partition of the components of $\partial M$. 

Consider a triplet $(W_0',W'_1,H')$ such that the following holds:
\begin{itemize}
\item the surface $H'$ is embedded inside the interior of $M$ and is obtained from $H$ by an isotopy of $M$ leaving $\partial M$ unchanged, 
\item $W'_0$, $W'_1$ are compression bodies included in $M$ so that $(W_0',W'_1,H')$ is a generalized Heegaard splitting of 
$$(W'_0 \cup W'_1, \partial_-W'_0, \partial_-W'_1),$$
\item $\partial_-W'_0$ and $\partial_-W'_1$ are two (not necessarily connected, possibly empty) oriented stable minimal surfaces.
\end{itemize}
Such a triplet $(W_0',W'_1,H')$ is then called an \textit{$H$-compatible splitting}. Note that $(W_0,W_1,H)$ is itself an $H$-compatible splitting. By the maximum principle, each component of $\partial_-W'_0 \cup \partial_-W'_1$ is either disjoint from $\partial M$ or entirely contained in it.

Given a metric $g$ on $M$, let $\mathcal{S}_g(H)$ be the set of subsets $N\subset M$ such that $N=W_0'\cup W'_1$ for a certain $H$-compatible splittings $(W_0',W'_1,H')$. The inclusion relation $\subset$ gives a partial order on $\mathcal{S}_g(H)$. An element $N$ is then minimal in the sense of this partial order (not to be confused with the minimality of the boundary $\partial N$) if there is no element $\tilde{N}$ such that $\tilde{N}\subset  N'$.

\begin{lemme} \label{max spheres} 
Let $(M,g)$ be as above and suppose that $H$ is an irreducible generalized Heegaard splitting. Then 
\begin{enumerate}
\item either $H$ is isotopic to a stable minimal surface, 
\item or $H$ is isotopic to the stable oriented double cover of a non-orientable minimal surface $\Sigma \subset M$ with a vertical handle attached, 
\item or for any $N\in\mathcal{S}_g(H)$, there exists a minimal element $N_{min}\in\mathcal{S}_g(H)$ such that 
$$N_{min}\subset  N.$$
\end{enumerate}

\end{lemme}

\begin{proof}

Let $N=N_0\in\mathcal{S}_g(H)$, and consider a sequence $\{N_i\}\subset \mathcal{S}_g(H)$ such that 
$$...\subset N_2 \subset N_1 \subset N_0 \text{ and}$$ 
\begin{equation*} \label{limvol}
\mathcal{H}^3(N_{i+1}) \leq  \inf\{\mathcal{H}^3(N') ; N' \subset N_i\}+\frac{1}{i+1}. 
\end{equation*}
We want to show that $\partial N_i$ converges subsequently to a minimal surface. By stability, \cite{Sc} gives a uniform upper bound $K$ on the second fundamental form of $\partial N_i$ and hence a uniform lower bound $\delta_0>0$ on the local focal distance of $\partial N_i$ (the supremum of the $d$ such that, given $\nu$ a unit normal of $S_i$, the map 
$$\partial N_i \times (-d,d)Ê\to M$$
$$(x,s) \mapsto \exp_x(s\nu)$$
is locally a diffeomorphism). It suffices now to rule out the possibility that the area $\mathcal{H}^2(\partial N_i)$ is unbounded. Note that in our case the sequence of stable minimal surfaces is nested, i.e. for each $\partial N_i$ the other $\partial N_j$ are on one side of $\partial N_i$. Moreover since $\mathcal{H}^3(N_{i})$ is a converging sequence, for any $\mu>0$, there exists an integer $i_0$ such that the volume of $N_i \backslash N_j$ is smaller than $\mu$ if $j\geq i \geq i_0$. Let $\nu$ be the unit normal of $\partial N_i$ pointing inside of $N_i$. If $\mu$ is chosen small compared to $K$ then for any $j\geq i_0$, $x\in \partial N_{i_0}$, $\{\exp_x(s\nu); s\in[0,\delta_0]\} \cap (\partial N_{i_0} \cup \partial N_j) \neq \varnothing$. So since $N_{i_0} \subset N_j $, $\partial N_j$ is locally a graph over $\partial N_{i_0}$ with bounded slope, and we conclude that the area $\mathcal{H}^2(\partial N_j)$ ($j\geq i_0$) is bounded in terms of $ \mathcal{H}^2(\partial N_{i_0})$, as desired. The sequence $\partial N_i$ converges subsequently to an embedded minimal surface $\Sigma$. When the limit of $\mathcal{H}^3(N_{i})$ is zero then either $\Sigma$ is an oriented stable minimal surface isotopic to $H$ or $\Sigma$ is a connected non-orientable minimal surface with stable double cover bounding a compression body on one side. Then by irreducibility, the generalized Heegaard splitting $H$ is isotopic to this double cover with a vertical handle attached \cite{Heath}. When $\lim_{i\to \infty}\mathcal{H}^3(N_{i})$ is not zero, $\Sigma$ bounds the desired minimal element $N_{min}$ of $\mathcal{S}_g(H)$.

\end{proof}

It will be convenient to introduce this definition:

\begin{definition} \label{coredef}
Let $M$ be as above and $H$ be a Heegaard splitting. An \textit{$H$-core} of $M$ is a minimal element $\mathfrak{C}$ of $\mathcal{S}_g(H)$. 

Let $C$ be a positive real number. A \textit{$C$-bounded $H$-core} of $M$ is an element $\mathfrak{C}$ of $\mathcal{S}_g(H)$ such that $\mathcal{H}^2(\partial \mathfrak{C})\leq C$, and which is minimal among elements $N$ of $\mathcal{S}_g(H)$ satisfying $\mathcal{H}^2(\partial N)\leq C$.

\end{definition}

We can reformulate Lemma \ref{max spheres} as follows:

\begin{coro} \label{existence core}

Let $(M,g)$ be as above, $H$ is an irreducible Heegaard splitting. Suppose that $H$ is 
\begin{itemize}
\item neither isotopic to a stable minimal surface 
\item nor isotopic to the stable oriented double cover of a non-orientable minimal surface $\Sigma \subset M$ with a vertical handle attached. 
\end{itemize}
Then there exists an $H$-core $\mathfrak{C}$ of $M$. 

Similarly for all $C>Area(\partial M)$, let $H$ be a strongly irreducible Heegaard splitting and suppose that $H$ is 
\begin{itemize}
\item neither isotopic to a stable minimal surface of area at most $C$
\item nor isotopic to the boundary of a tubular neighborhood of a non-orientable minimal surface $\Sigma \subset M$ with a vertical handle attached, where $\Sigma$ has stable oriented double cover and area at most $C$.
\end{itemize}
Then there exists a $C$-bounded $H$-core $\mathfrak{C}$ of $M$.

\end{coro}

\begin{proof}
For the fist part, it suffices to apply Lemma \ref{max spheres} with $N=M$.

The second part is easier to prove since the stable minimal surfaces in consideration have area bounded by $C$ and so it follows by usual convergence arguments \cite{Sc}.
\end{proof}

\section{Local min-max theory} \label{prelim}

\subsection {} \textbf{Min-max constructions in the Almgren-Pitts setting}

We first explain the local min-max theory in the Almgren-Pitts setting. Even though the min-max theorem of this subsection (Theorem \ref{discreteminmax}) will not be used for our main applications in Section \ref{Rubinstein conj}, the statement is similar to the main min-max result (Theorem \ref{smoothminmax}) in the next subsection and the proof is much easier. Moreover it will be used for proving Theorem \ref{boomchak2} in Section \ref{R>0}. For these reasons, we decide to present it here.

We give a review of the basic definitions from Geometric Measure Theory and some notions of the Almgren and Pitts' theory used thereafter. For a complete presentation, we refer the reader to the book of Pitts \cite{P} and to Section 2 in \cite{MaNeinfinity}.

Let $M$ be a closed connected Riemannian $(n+1)$-manifold, assumed to be isometrically embedded in $\mathbb{R}^P$. We work with the space $\mathbf{I}_k(M)$ of $k$-dimensional integral currents with support contained in $M$, the subspace $\mathcal{Z}_k(M) \subset \mathbf{I}_k(M)$ whose elements have no boundary, and with the space $\mathcal{V}_k(M)$ of the closure, in the weak topology, of the set of $k$-dimensional rectifiable varifolds in $\mathbb{R}^P$ with support in $M$.

An integral current $T\in \mathbf{I}_k(M)$ determines an integral varifold $|T|$ and a Radon measure $||T||$ (\cite[Chapter 2, 2.1, (18) (e)]{P}). If $V \in \mathcal{V}_k(M)$, denote by $||V||$ the associated Radon measure on $M$. Given an $(n+1)$-dimensional rectifiable set $U\subset M$, if the associated rectifiable current (oriented as $M$) is an integral current in $\mathbf{I}_{n+1}(M)$, it will be written as $[|U|]$. To a rectifiable subset $R$ of $M$ corresponds an integral varifold called $|R|$. 
The support of a current or a measure is denoted by $\spt$. 
The notation $\mathbf{M}$ stands for the mass of an element in $\mathbf{I}_k(M, \mathbb{Z})$. On $\mathbf{I}_k(M)$ there is also the flat norm $\mathcal{F}$ which induces the so-called flat topology. 
The space $\mathcal{V}_k(M)$ is endowed with the topology of the weak convergence of varifolds. The mass of a varifold is denoted by $\mathbf{M}$. The $\mathbf{F}$-metric was defined in \cite{P} and induces the varifold weak topology on any subset of $\mathcal{V}_k(M)$ with mass bounded by a constant.

We denote $[0,1]$ by $I$. For each $j\in \mathbb{N}$, $I(1,j)$ stands for the cell complex on $I$ whose $1$-cells and $0$-cells are, respectively,
$$[0,3^{-j}], [3^{-j},2.3^{-j}], ... , [1-3^{-j},1] \text{  and  } [0],[3^{-j}], ..., [1-3^{-j}],[1].$$
$I(1,j)_p$ denotes the set of all $p$-cells in $I(1,j)$.


\begin{definition}
Whenever $\phi : I(1,j)_0 \to \mathcal{Z}_n(M)$, we define the fineness of $\phi$ as
$$\mathbf{f}(\phi) = \sup \bigg\{ \frac{\mathbf{M}(\phi(x) - \phi(y))}{\mathbf{d}(x,y)} ; x,y \in I(1,j)_0,x\neq y\bigg\}$$
where $\mathbf{d}(x,y) = 3^j|x-y|.$
\end{definition}

For each $x\in I(1,j)_0$, define $\mathbf{n}(i,j)(x)$ as the unique element of $I(1,j)_0$ such that $$\mathbf{d}(x,\mathbf{n}(i,j)(x)) = \inf \{\mathbf{d}(x,y) ; y\in I(1,j)_0\}.$$ 
We can now define discrete sweepouts. Let $C_0,C_1\in\mathcal{Z}_n(M)$.

\begin{definition}
\begin{enumerate}

\item Let $\delta>0$. We say that $\phi_1$ and $\phi_2$ are homotopic in $(\mathcal{Z}_n(M),C_0,C_1)$ with fineness $\delta$ if and only if there exist positive integers $k_1$, $k_2$, $k_3$ and a map
$$\psi : I(1,k_3)_0 \times I(1,k_3)_0 \to \mathcal{Z}_n(M)$$
such that $\mathbf{f}(\psi)<\delta$ and whenever $j=1,2$ and $x\in I(1,k_3)_0$,
$$\phi_j : I(1,k_j) \to \mathcal{Z}_n(M), \quad \phi_j ([0])= C_0 \quad \phi_j([1]) =C_1,$$
$$\psi([j-1],x) = \phi_j(\mathbf{n}(k_3,k_j)(x)), \quad \psi(x,[0]) =C_0 \quad  \psi(x,[1]) = C_1.$$

\item A homotopy sequence of mappings into $(\mathcal{Z}_n(M),C_0,C_1)$ is a sequence $S=\{\phi_1,\phi_2,...\} $ for which there exist positive numbers $\delta_1,\delta_2...$ such that $\phi_i$ is homotopic to $\phi_{i+1}$ in $(\mathcal{Z}_n(M),C_0,C_1)$ with fineness $\delta_i$ for each $i\in \mathbb{N}^*$, $\lim_i\delta_i =0$ and 
$$\sup\{\mathbf{M}(\phi_i(x)) ; x \in \dmn(\phi_i), i\in \mathbb{N}\backslash\{0\}\} <\infty.$$

\item If $S_1=\{\phi_i^1\}$ and $S_2=\{\phi_i^2\}$ are homotopy sequences of mappings into $(\mathcal{Z}_n(M),C_0,C_1)$, then $S_1$ is homotopic with $S_2$ if and only if there is a sequence of positive numbers $\delta_1,\delta_2,...$ such that $\lim_i \delta_i =0$ and $\phi_i^1$ is homotopic to $\phi_i^2$ in $(\mathcal{Z}_n(M),C_0,C_1)$ with fineness $\delta_i$ for each positive integer $i$.

"To be homotopic with" is an equivalence relation on the set of homotopy sequences of mappings into $(\mathcal{Z}_n(M),C_0,C_1)$. An equivalence class of such sequences is called a homotopy class of mappings into $(\mathcal{Z}_n(M),C_0,C_1)$. The space of these equivalence classes is denoted by $$\pi_1^{\sharp}(\mathcal{Z}_n(M),C_0,C_1).$$

\end{enumerate}
\end{definition}

\begin{remarque} Compared with Pitts \cite{P}, we need a more flexible definition of sweepouts, whose ends are allowed to be arbitrary cycles in $\mathcal{Z}_n(M)$. This will enable to localize the min-max theory.

\end{remarque}


Let us define the local min-max width. Let $(M^{n+1},g)$ be a closed oriented manifold and let $N$ be a compact non-empty $(n+1)$-submanifold with boundary. Suppose that $\Gamma_0$ and $\Gamma_1$ are disjoint closed sets, $\Gamma_0 \cup \Gamma_1 = \partial N$, and $C_0$ (resp. $C_1$) is the cycle in $\mathcal{Z}_n(N)$ which is determined by $\Gamma_0$ (resp. $\Gamma_1$) with multiplicity one and with orientation given by the inward (resp. outward) unit normal. Given $\Pi \in \pi_1^{\sharp}(\mathcal{Z}_n(N),C_0,C_1)$, consider the function $\mathbf{L}: \Pi \to [0,\infty]$ defined such that if $S=\{\phi_i\}_{i \in \mathbb{N}} \in \Pi$, then 
$$\mathbf{L}(S) = \limsup_{i\to \infty} \max\{\mathbf{M}(\phi_i(x)) ; x \in \dmn(\phi_i)\}.$$
The width of $\Pi$ is then the following quantity:
$$\mathbf{L}(\Pi) = \inf\{\mathbf{L}(S) ; S \in \Pi\}.$$
A sequence $S=\{\phi_i :I(1,n_i) \to \mathcal{Z}_n(N)\}_{i} \in \Pi$ is said to be critical for $\Pi$ if $\mathbf{L}(S)=\mathbf{L}(\Pi)$. Define the critical set $\mathbf{C}(S) \subset \mathcal{V}_n(N)$ of $S \in \Pi$ as
\begin{align*}
\mathbf{C}(S) = \{V ;  \quad& \exists \{i_j\}_j, \exists \{x_j\}, i_j \to \infty, x_j \in \dmn(\phi_{i_j}) ,\\  
& V = \lim\limits_{j\to \infty} |\phi_{i_j}(x_j)| \text{ and } ||V||(N)=\mathbf{L}(S)\}.
\end{align*}

In \cite{Alm1}, Almgren describes how to associate to a map $\phi: I(1,j)_0 \to \mathcal{Z}_n(M)$, with fineness small enough, an element of $\mathbf{I}_{n+1}(M)$. Let us explain this construction. There is a number $\mu>0$ such that if $T\in \mathbf{I}_n(M)$ has no boundary and $\mathcal{F}(T)\leq \mu$, then there is an $S\in \mathbf{I}_{n+1}(M)$ such that $\partial S =T$ and $$\mathbf{M}(S)=\mathcal{F}(T)=\inf\{\mathbf{M}(S') ; S'\in \mathbf{I}_{n+1}(M) \text{ and } \partial S' = T \}.$$
Such an $S$ is called an $\mathcal{F}$-isoperimetric choice for $T$. A chain map 
$$\Phi : I(1,j)\to \mathbf{I}_*(M)$$
of degree $n$ is a graded homomorphism $\Phi $ of degree $n$, such that $\partial \circ \Phi = \Phi \circ \partial$ and $\Phi(\alpha)$ is $\mathcal{F}$-isoperimetric for $\Phi(\partial \alpha)$ where $\alpha \in I(1,j)_1$. Now let $\phi: I(1,j)_0 \to \mathcal{Z}_n(M)$, be a map with fineness smaller than $\mu$. There is a chain map $\Phi : I(1,j)\to \mathbf{I}_*(M)$ such that 
$$\Phi(x)=\phi(x) \quad \forall x \in I(1,j)_0.$$
Consider the $(n+1)$-dimensional integral current
\begin{equation} \label{Almgren}
\sum_{\alpha \in I(1,j)_1} \Phi(\alpha).
\end{equation}
By the interpolation formula of \cite[Section 6]{Alm1}, this sum is invariant by homotopies with fineness smaller than $\mu$. Hence when $\Pi \in \pi_1^{\sharp}(\mathcal{Z}_n(N),C_0,C_1)$ and $\{\phi_i\}\in\Pi$, the map which associates to $\Pi$ the $(n+1)$-dimensional current (\ref{Almgren}), defined with $\phi_i$ for $i$ sufficiently large, is well defined. We call this map the Almgren map and we denote it by
$$\mathcal{A} :  \pi_1^{\sharp}(\mathcal{Z}_n(N),C_0,C_1) \to \mathbf{I}_{n+1}(N).$$

\begin{remarque}
 
\item We will use interpolation results which make it possible to get discrete sweepouts from maps continuous in the flat topology and vice-versa. These results were proved in \cite[Sections 13-14]{MaNeWillmore}, were extended in \cite{Zhou} and \cite{Zhou2} confirmed that the technical condition "no mass concentration" was actually superfluous. 

\end{remarque}

The mean curvature of a surface endowed with the outward normal $\nu$ will be said to be positive when the mean curvature vector is a negative multiple of $\nu$. Piecewise smooth mean convexity was defined in \cite[Definition 10]{Antoine}. In what follows, given a set $E$, its closure in the ambient space will be denoted by $\bar{E}$. Its topological boundary will be $\partial E$. When $E$ is an open set with regular boundary, $\partial E$ inherits a natural orientation and outward normal $\nu$. The metric $g$ is said to be bumpy if no smooth immersed closed minimal hypersurface has a non-trivial Jacobi vector field. White showed that bumpy metrics are generic in the Baire sense \cite{Whitebumpy,Whitebumpy2}.

\begin{theo} \label{discreteminmax}

Let $(M^{n+1},g)$ be a closed oriented manifold endowed with a bumpy metric $g$ and $2\leq n \leq 6$. Let $N$ be a compact non-empty $(n+1)$-submanifold of $M$ whose boundary components are either piecewise smooth mean convex or minimal hypersurfaces. Suppose that $\Gamma_0$ and $\Gamma_1$ are disjoint closed sets, $\Gamma_0 \cup \Gamma_1 = \partial N$, and $C_0$ (resp. $C_1$) is the cycle in $\mathcal{Z}_n(N)$ which is determined by $\Gamma_0$ (resp. $\Gamma_1$) with multiplicity one and with orientation given by the inward (resp. outward) unit normal.

Suppose that the homotopy class $\Pi \in \pi^\sharp (\mathcal{Z}_n(N), C_0,C_1)$ satisfies
$$ \mathbf{L}(\Pi) >\max (\mathbf{M}(C_0) , \mathbf{M}(C_1)\}.$$
Then there exists a stationary integral varifold $V$ whose support is a smooth embedded minimal hypersurface $\Sigma$ of index bounded by one, such that
$$||V||(N)= \mathbf{L}(\Pi).$$ Moreover one of the component of $\Sigma$ is contained in the interior $\interior(N) $. 
\end{theo}

\begin{proof}
Let $\Upsilon_u$ (resp. $\Upsilon_s$) be the union of minimal hypersurfaces in $\partial N$ which are unstable (resp. stable). Since the metric is bumpy, we can find a small $\delta>0$ so that 
$N_\delta := (N \cup \{x\in M; d(x,\Upsilon_s)\leq \delta\}) \backslash \{d(x,\Upsilon_u)<\delta\}$
is a strictly mean convex domain and if a closed minimal hypersurface is contained in $N_\delta$ then it is contained in $N$. The homotopy class $\Pi$ naturally induces another homotopy class $\Pi_\delta$ associated with $N_\delta$. It is not difficult to check that for $\delta$ small, $\mathbf{L}(\Pi_\delta)=\mathbf{L}(\Pi)$. We fix such a $\delta>0$. As in \cite[Theorem 2.1]{MaNe}, \cite[Theorem 2.7]{Zhou}, we get the existence of $\Sigma$ and the index bound follows from Theorem 1.7 in \cite{MaNe} (which is applicable in our context with boundary). The same result implies that at least one component of the min-max surface is inside $\interior(N)$, because the elements of $\Upsilon_s$ are stable two-sided. 

\end{proof}

The next lemma gives a criterion to apply the previous theorem. It was essentially proved in \cite{MorganRos} but we give another proof for the sake of completeness.

\begin{lemme} \label{canapply1}
Assume that ${N}\subset M^{n+1}$, g, $\Gamma_0$, $\Gamma_1$, $C_0$, $C_1$ are as in Theorem \ref{discreteminmax}. Consider ${\Pi} \in  \pi^\sharp (\mathcal{Z}_n(N), C_0,C_1)$ such that $\mathcal{A}({\Pi}) =[|N|]$. Suppose that $\Gamma_0$ is a stable minimal hypersurface. Then 
$$ \mathbf{L}(\Pi) > \mathbf{M}(C_0).$$
Consequently, if $\Gamma_0$ and $\Gamma_1$ are both stable minimal hypersurfaces, then the conclusion of Theorem \ref{discreteminmax} holds.
\end{lemme}

\begin{proof}

Denote by $|\Gamma_0|$ the varifold determined by $\Gamma_0$ with multiplicity one. Since any homotopy sequence $\{\psi_i\}\in \Pi$ sweeps out $N$ non trivially, for all $\epsilon_0>0$ small enough there is an element $x\in I(1,k_i)_0$ for which $\mathbf{F}(|\psi_i(x)|, |\Gamma_0|) \in (\epsilon_0,2\epsilon_0)$ when $i$ is large enough. Let $\mathbf{B}^{\mathbf{F}}_{\epsilon}(|\Gamma_0|)$ be the $\mathbf{F}$-ball of radius $\epsilon$ in $\mathcal{V}_2(M)$ centered at $|\Gamma_0|$.

\textbf{Claim:} For $\epsilon_0$ small enough, there is a $\delta>0$ such that for any integral cycle $T\in\mathcal{Z}_2(M)$ satisfying 
$$|T|\in A:= \mathbf{B}^{\mathbf{F}}_{2\epsilon_0} (|\Gamma_0|)\backslash \mathbf{B}^{\mathbf{F}}_{\epsilon_0}(|\Gamma_0|),$$
we have $\mathbf{M}(T) \geq \mathbf{M}(C_0) +\delta$.

The lemma readily follows from this claim. To prove the latter, we argue by contradiction and consider a sequence of cycles $T_i\in\mathcal{Z}_2(M)$ with $|T_i|\in A$, and a sequence of positive numbers $\delta_i$ going to zero such that $ \mathbf{M}(T_i) \leq \mathbf{M}(C_0) +\delta_i$. Let $\Omega_r$ be an $r$-neighborhood of $\Gamma_0$ so that there is a family of area-decreasing maps $\{P_t\}_{t\in[0,1]}$ as in \cite[Proposition 5.7]{MaNeindexbound}. Note that $||T_i||(M\backslash \Omega_{r/2})$ is smaller than $\kappa.\epsilon_0$ where $\kappa=\kappa(\Gamma_0,r)\geq 1$ is a constant. By the properties of $\Omega_r$, if we fix $\epsilon_0$ small enough then for any integral cycle $\hat{T}$ with $|\hat{T}| \in \mathbf{B}^{\mathbf{F}}_{(1+2\kappa)\epsilon_0}(|\Gamma_0|)$ and support in $\bar{\Omega}_r$ we have by the constancy theorem:
\begin{equation} \label{constancy}
\quad (P_1)_\sharp \hat{T} = \pm C_0 \text{ and }\mathbf{M}(C_0)\leq \mathbf{M}(\hat{T}).
\end{equation}
For almost all $r'\in(r/2,r)$, we can minimize the part of $T_i$ outside $\Omega_{r'}$, by the monotonicity formula (fix $\epsilon_0$ small) we get an integral cycle $T_i'$ coinciding with $T_i$ inside $\Omega_{r'}$ but area-minimizing outside $\bar{\Omega}_{r'}$, and satisfying 
$$\spt(T_i') \subset \Omega_r \text{ and } \mathbf{M}(T_i')\leq \mathbf{M}(T_i).$$ 
By the choice of the sequence $T_i$, 
\begin{equation} \label{intuitif}
\mathbf{M}(C_0)\leq \mathbf{M}(T_i')\leq \mathbf{M}(T_i) \leq \mathbf{M}(C_0) +\delta_i.
\end{equation}
Note that by construction $\mathbf{F}(|T_i|,|T_i'|) \leq 2 ||T_i||(M\backslash \Omega_{r'})$. Since $\delta_i$ goes to zero, $r'=r'(i)\in(r/2,r)$ can be chosen so that the mass $ ||T_i||(M\backslash \Omega_{r'}) $ also converges to zero. Indeed, either $||T_i||(\Omega_{r}\backslash\Omega_{r/2})$ goes to zero or not. In the first case, let $f$ be the function defined before Proposition 5.7 in \cite{MaNeindexbound}. By the coarea formula, we find $r'\in(r/2,r)$ such that $\langle T_i, f, r'\rangle $ is an integral current with arbitrarily small mass: consequently $||T_i'||(M) -||T_i||(\Omega_{r'})$ is arbitrarily small. Then $||T_i||(\Omega_{r'})$ is arbitrarily close to $\mathbf{M}(C_0)$ because the integral cycle $T_i'$ satisfies (\ref{intuitif}). Since $\mathbf{M}(T_i) = ||T_i||(\Omega_{r'}) +||T_i||(M\backslash \Omega_{r'})$, it forces $||T_i||(M\backslash\Omega_{r'})$ to go to zero too. In the second case, namely if $||T_i||(\Omega_{r}\backslash\Omega_{r/2})$ does not go to zero, we choose $r'$ tending to $r$ as $i\in \infty$, such that $||T'_i||(\Omega_{r}\backslash\Omega_{r/2})$ is also bounded away from zero (for a subsequence in $i$). Then by the computation in the proof of \cite[Proposition 5.7]{MaNeindexbound} and supposing (\ref{intuitif}) true, the derivative of $(P_t)_\sharp|T_i'|$ is uniformly bounded above by a negative constant for $t\in[0,t_0]$ where $t_0 >0$ is independent of $i$. This contradicts the upper bound in (\ref{intuitif}). To sum up, we just showed that for a certain choice of $r'=r'(i)$, we have $\lim \mathbf{F}(|T_i|,|T_i'|)=0$.

Besides, by compactness for integral cycles, $T_i'$ converges subsequently to a current $T'_\infty$ in the flat topology with mass at most $\mathbf{M}(C_0)$ and $(P_1)_\sharp T'_i$ is equal to $\pm C_0$ for $i$ large, by (\ref{constancy}). Hence, since $T'_\infty$ has support in $\Gamma_0$ by \cite[Proposition 5.7, (iv)]{MaNeindexbound}, $T'_\infty=\pm C_0$ and there is no loss of mass so $|T'_i|$ converges to $|\Gamma_0|$ as varifolds. In conclusion, $\mathbf{F}(|T_i'|,|\Gamma_0|)$ goes to zero, and 
$$\mathbf{F}(|T_i|,|\Gamma_0|) \leq \mathbf{F}(|T_i|,|T_i'|)+ \mathbf{F}(|T_i'|,|\Gamma_0|)$$
also converges to zero, contradicting the fact that $T_i \in A$.

\end{proof}

\subsection{}\textbf{Min-max constructions in the Simon-Smith setting}

Let $N$ be an oriented connected compact $3$-manifold possibly with boundary, subset of a closed oriented manifold $(M,g)$. The surfaces considered in this subsection are all embedded. Since $M$ is oriented, we will only consider smooth sweepouts $\{\Sigma_t\}$ where all the slices are oriented. If $\Sigma$ is a surface, we denote by $|\Sigma|$ the varifold associated to $\Sigma$ with multiplicity one. 

\begin{definition} \label{defsmoothsweep}
Let $\{\Sigma_t\}_{t\in[a,b]}$ be a family of oriented closed surfaces in $N$. We say that $\{\Sigma_t\}$ is a  \textit{smooth sweepout} if 
\begin{enumerate}
\item for all $t\in [a,b]$, $\Sigma_t$ is a smooth surface in the interior of $N$,
\item $\Sigma_t$ varies smoothly in $t\in (a,b)$,
\item there is a partition $(A,B)$ of the components of $\partial N$ such that, $\Sigma_a=A$, $\Sigma_b=B$, and $|\Sigma_t|$ converges to $|\Sigma_a|$ (resp. $|\Sigma_b|$) in the $\mathbf{F}$-norm as $t\to a$ (resp. $b$).
\end{enumerate}
\end{definition}


Let $\Pi$ be a collection of smooth sweepouts parametrized by $[0,1]$. Denote by $\Diff_0$ the set of diffeomorphisms of $N$ isotopic to the identity map and leaving the boundary fixed. The set $\Pi$ is saturated if for any map $\psi\in C^\infty([0,1]\times N, N)$ such that $\psi(t,.)\in \Diff_0$ for all $t$ and $\psi(0,.)=\psi(1,.)=\Id$, and for any $\{\Sigma_t\}_{[0,1]}\in \Pi$, we have $\{\psi(t,.)(\Sigma_t)\}_{[0,1]}\in \Pi$. We say that $\Pi$ is generated by a family of smooth sweepouts $\{\{\Sigma^j_t\}_{t\in [0,1]}\}_{j\in J}$ if $\Pi$ is the smallest saturated set containing $\{\{\Sigma^j_t\}_{t\in [0,1]}\}_{j\in J}$. The \textit{width of $N$ associated with $\Pi$} in the sense of Simon-Smith is defined to be 
$$W(N,\Pi) =\inf_{\{\Sigma_t\} \in \Pi}\sup_{t\in[0,1]} \mathcal{H}^2(\Sigma_t).$$
Given a sequence of smooth sweepouts $\{\{\Sigma^i_t\}\}_i\subset \Pi$, denote by $\mathbf{\Lambda}(\{\{\Sigma^i_t\}\}_i)$ the set 
\begin{align*}
\{V\in \mathcal{V}_2(M) ;\quad & \exists  \{i_j\} \to \infty, t_{i_j}\in [0,1] \\ 
\text{ such that }& \lim_{j\to \infty} \mathbf{F}(|\Sigma^{i_j}_{t_{i_j}}|, V) =0|) \}.
\end{align*}
The sequence $\{\Sigma^i_t\} \in\Pi$ is \textit{minimizing} if $\lim_{i\to \infty} \max_{t\in[0,1]}\mathcal{H}^2(\Sigma^i_t)=W(N,\Pi)$, it is \textit{pulled-tight} if moreover any element $V\in \mathbf{\Lambda}(\{\{\Sigma^i_t\}\}_i)$ with $||V||(M)=W(N,\Pi)$ is stationary. By \cite[Proposition 4.1]{C&DL}, it is always possible to deform $\{\Sigma^i_t\} \in\Pi$ into a pulled-tight sequence.

The Almgren map $\mathcal{A}$ defined in the previous subsection can also be defined on the family of smooth sweepouts of $M$. The smooth sweepouts we consider have oriented slices so by choosing a continuous orientation, it determines a family of integral cycles continuous in the flat topology. Then by \cite{Alm1}, one can associate to it a $3$-dimensional  integral current $\mathcal{A}(\{\Sigma_t\})$.


The following theorem will be crucial for proving Rubinstein's conjecture. The point is to show that in the setting of Simon-Smith, under certain conditions, local min-max gives a minimal surface \textit{inside the interior} of the connected domain $N$.

\begin{theo} \label{smoothminmax}
Consider an oriented connected compact $3$-manifold $N\subset (M,g)$ possibly with boundary. Let $(W_0,W_1,H)$ be a generalized Heegaard splitting of $(M,\partial_0 M, \partial_1 M)$, where $(\partial_0 M, \partial_1 M)$ is a partition of the components of $\partial M$. Suppose that $(W_0', W_1',H')$ is an $H$-compatible splitting with $N=W_0' \cup W_1'$. We write $\partial_- W_0' =\Gamma_0$ and $\partial_+W_1' = \Gamma_1$.
Let $\Pi$ be a saturated set generated by all smooth sweepouts $\{\Sigma_t\}$ such that 
\begin{itemize}
\item $\mathcal{A}(\{\Sigma_t\})=[|N|]$,
\item for $i\in\{0,1\}$, $\Sigma_i = \Gamma_i $,
\item for all $t\in(0,1)$, the surface $\Sigma_t$ is isotopic to $H'$ inside $N$.
\end{itemize}
Suppose that the metric $g$ is bumpy, that $(W_0',W_1',H')$ is a strongly irreducible Heegaard splitting and that $\max \{\mathcal{H}^2(\Gamma_0), \mathcal{H}^2(\Gamma_1)\}< W(N,\Pi).$

Then there exists a min-max sequence $\Sigma^j_{t_j}$ converging to $\sum_{i=1}^k m_i\Sigma^\infty_i$ as varifolds, where $\Sigma^\infty_i\subset N$ are disjoint closed embedded connected minimal surfaces such that
$$ \sum_{i=1}^k m_i \mathcal{H}^2(\Sigma^\infty_i) = W(N,\Pi),$$
the index of $\bigcup_{i=1}^k \Sigma^\infty_i$ is bounded by one and there is an $i_0$ such that
$$\Sigma^\infty_{i_0}  \subset  \interior(N).$$
Moreover, the genus bound and the analysis of \cite{Ketgenusbound} hold true.
\end{theo}

\begin{proof}       
By definition of "$H$-compatible" (see Subsection \ref{corecore}), the boundary $\partial N = \Gamma_0\cup\Gamma_1$ is a stable minimal surface. Since the metric is bumpy, $\partial N$ is actually strictly stable and
\begin{equation} \label{boundtrue}
\max \{Area(\Gamma_0),Area(\Gamma_1)\}< W(N,\Pi).
\end{equation}
This follows from Lemma \ref{canapply1} and the general fact that by discretizing a smooth sweepout, the Almgren-Pitts width is at most the Simon-Smith width.

We can also find a small $\hat{\delta}>0$ so that 
$$N_{\hat{\delta}} := (N \cup \{x\in M; d(x,\partial N)\leq \hat{\delta}\}) $$
is a strictly mean convex domain and if a closed minimal surface is contained in $N_{\hat{\delta}}$ then it is contained in $N$. The saturated set $\Pi$ naturally induces a saturated set $\Pi_{\hat{\delta}}$ associated with $N_{\hat{\delta}}$. It is then not difficult to check that for $\hat{\delta}$ small, $W(N_{\hat{\delta}},\Pi_{\hat{\delta}})=W(N,\Pi)$. If $\hat{\delta}$ is chosen small enough, by (\ref{boundtrue}), we can apply the version of the Simon-Smith theorem proved in \cite[Theorem 2.1]{MaNe} to get the existence of the varifold $V=\sum_{i=1}^k m_i\Sigma^\infty_i$, then the genus bound and the nature of convergence follow from \cite{Ketgenusbound}. The index of the union $\bigcup_{i=1}^k\Sigma^\infty_i$ is bounded by one according to Theorem 6.1 and paragraph 1.3 in \cite{MaNeindexbound}.

The goal of the remaining of the proof is to show the existence of a component $\Sigma_{i_0}^\infty$ inside the interior $\interior(N)$. The arguments will share similarities with \cite[Deformation Theorem C]{MaNeindexbound} (in particular the several constructions in its proof), however we have to deal with smooth isotopies. On the other hand, in our case we only need to rule out the case where the whole min-max surface is included in the boundary. Let $\{\{\Sigma^i_t\}\}_i$ be a pulled-tight minimizing sequence and suppose by contradiction that for all $V\in \mathbf{\Lambda}(\{\{\Sigma^i_t\}\}_i)$ with smooth support and mass $W(N,\Pi)$, $\spt(V)$ is included in $\partial N$. 
Given a sweepout in $\Pi$, we orientate $\Sigma_t$ with the unit normal $\nu$ pointing towards $\Gamma_1$. In $N$, each $\Sigma_t$ hence bounds a manifold with boundary $B(\Sigma_t)$ such that $\nu$ is the outward normal.

If $S$ is a surface, let $|S|$ be the varifold it determines with multiplicity one. Denote by $S_1$, ..., $S_p$ the stable minimal components of $\Gamma_0\cup \Gamma_1$. Let $V$ be a varifold with mass $W(N,\Pi)$, of the form:
\begin{equation} \label{form}
V = m_1|S_1|+...m_p |S_p|,
\end{equation}
where $m_i$ are nonnegative integers. We call the finite family of such $V$ by $\hat{\mathcal{V}} $. Let $\{\Sigma_t\}_{t\in[0,1]}$ be a smooth sweepout in $\Pi$. We are applying the following discussion to the pulled-tight minimizing sequence $\{\{\Sigma^i_t\}\}_i$ so we are assuming $\{\Sigma_t\}$ to be one of these sweepouts. We can in particular make $\max_t\mathcal{H}^2(\Sigma_t)- W(N,\Pi)$ arbitrarily small. Given $\alpha>0$, consider $\mathbf{V}_{\alpha}$ the subset 
$$\mathbf{V}_{\alpha}:=\{t\in[0,1] ; \quad \exists V \in \hat{\mathcal{V}} , \quad\mathbf{F}(|\Sigma_t|, V) \leq \alpha\}.$$

The idea of the end of the proof is to construct from part of the sweepout $\{\Sigma_t\}$ another sweepout $\{\hat{\Sigma}_t\}$ (not necessarily homotopic to $\{\Sigma_t\}$) for which the corresponding $\mathbf{V}_{\eta}$ is empty, where $\eta>0$ is a constant depending on $\alpha$ and $N$. A technical point which is already in Claims 1-4 of the proof of \cite[Deformation Theorem C]{MaNeindexbound} is that we will make sure that the surfaces $\hat{\Sigma}_t$ are not close to any stationary integral varifold which was far from the surfaces $\Sigma_t$. We suppose that $\mathbf{V}_{\alpha}$ is non-empty (otherwise there is nothing to prove) and is a finite union of closed intervals. If $\alpha$ is sufficiently small, then for any $t\in \mathbf{V}_{\alpha}$, $\Sigma_t$ bounds $B(\Sigma_t)$ which has volume either close to $0$ or close to $\Vol(N)$. Let $[a_1,b_1]$,...,$[a_q,b_q]$ be the intervals in $\mathbf{V}_{\alpha}$ such that $B(\Sigma_t)$ has volume close to $0$, where $a_1\leq b_1<...<a_q \leq b_q$. We have $b_q<1$ since $\alpha$ is small enough. Similarly let $[c_1,d_1]$,...,$[c_{q'},d_{q'}]$ be the intervals composing $\mathbf{V}_{\alpha} \cap (b_q,1]$. On these intervals the volume of $B(\Sigma_t)$ is close to $\Vol(N)$. Then $\{\hat{\Sigma}_t\}$ will be constructed from the restriction $\{\Sigma_t\}_{t\in [b_q,c_1]}$ by appropriately closing its ends, i.e. by deforming $\Sigma_{b_q}$ and $\Sigma_{c_1}$ respectively to $\Gamma_0$ and $\Gamma_1$. This new sweepout $\{\hat{\Sigma}_t\}$ will not necessarily be homotopic to $\{\Sigma_t\}$ (contrarily to the analogue situation in \cite[Deformation Theorem C]{MaNeindexbound}).

We recall the following version ``with constraints'' of the $\gamma$-reduction of \cite{MSY} used in the min-max setting of Simon-Smith \cite{Smith} (see \cite[Section 7]{C&DL}). Consider $\tilde{\Sigma}$ a surface embedded in $\interior(N)\cap N_\delta$. Let $U$ be an open set included in $N$. Let $\mathfrak{Is}(U)$ be the set of isotopies of $N_\delta $ fixing $N_\delta\backslash U$, with parameter in $[0,1]$, and for $\mu>0$ define
$$\mathfrak{Is}_{\mu}(U) = \{\psi\in\mathfrak{Is}(U); \mathcal{H}^2(\psi(\tau,\tilde{\Sigma})) \leq \mathcal{H}^2(\tilde{\Sigma}) + \mu \text{ for all } \tau \in[0,1]\}.$$
An element of the above set is called a $\mu$-isotopy. Suppose that the sequence $\{\psi^k\}\subset \mathfrak{Is}_{\mu}(U)$ is such that 
$$\lim_{k\to\infty}\mathcal{H}^2(\psi^k(1,\tilde{\Sigma})) = \inf_{\psi\in\mathfrak{Is}_{\mu}(U)}\mathcal{H}^2(\psi(1,\tilde{\Sigma})).$$
Then in $U$, $\psi^k(1,\tilde{\Sigma})$ subsequently converges in the varifold sense to a smooth minimal surface $\tilde{\Gamma}$. 

Let $\Omega_r$ be an $r$-neighborhood of $S_1\cup ...\cup S_p \subset \Gamma_0\cup \Gamma_1$ so that there is a family of area-decreasing maps $\{P_t\}_{t\in[0,1]}$ as in \cite[Proposition 5.7]{MaNeindexbound} ($r>0$ is small). Let $\alpha_1>0$ be such that if for $V\in \hat{\mathcal{V}}$, a stationary integral varifold $Z$ has $\spt(Z)\subset \spt(V)$, $\mathbf{M}(Z)=\mathbf{M}(V)$ but $Z\neq V$, then there is a connected component $\Omega^Z$ of $\Omega_{r}$ so that $||Z||(\Omega^Z) > ||V||(\Omega^Z) +\alpha_1$. The $\alpha$ considered in the previous paragraph will then be taken independent of $\{\Sigma_t\}$ and at least small enough so that the following holds: if for $V\in \hat{\mathcal{V}}$, a varifold $Z'\in \mathcal{V}_2(M)$ (not necessarily stationary or integral) has $||Z'||(\Omega^{Z'}) > ||V||(\Omega^{Z'}) +\alpha_1$ for a connected component $\Omega^{Z'}$ of $\Omega_{r}$, then $\mathbf{F}(Z,V)>2\alpha$ (see \cite[Subsections 5.11 and 5.14]{MaNeindexbound}).
Let $\mu>0$ be a small positive constant which depends on $\{\Sigma_t\}$ and which will be determined later.

Notice (see Remark \ref{topolo} after this proof) that by strong irreducibility of $\Sigma_{b_q}$ in $N$, during a $\gamma$-reduction, there are a finite number of surgeries along curves and at each step, if it is done along an essential curve (in the surface of this step) then the surgery disk is on one well-determined side of $\Sigma_{b_q}$ independent of the surgery (the side of $\Gamma_0$ in the case when $B(\Sigma_{b_q})$ is close to $0$). Surgeries along non-essential curves can occur on both sides and split off spheres. Moreover the non-sphere components have multiplicity one. Let us apply the bounded path version of the $\gamma$-reduction to $\Sigma_{b_q}$, with $U=\interior(N_{\delta})\backslash \bar{\Omega}_{r/2}$, with $\mu$ small to be determined later. If $\alpha$ is small enough, then by the monotonicity formula the resulting limit varifold $V_1$ has support in ${\Omega}_r$. Hence $\Sigma_{b_q}$ can be deformed by a $\mu$-isotopy to a surface $A$, obtained by attaching to a closed surface $B \subset {\Omega}_r$ some thin handles (diffeomorphic to $S^1\times[0,1]$) and caps (diffeomorphic to disks), with $\mathcal{H}^2((A\backslash B) \cup (B\backslash A))$ arbitrarily small (see \cite[Lemma 7.1]{Montezuma1} an analogue lemma in the context of Almgren-Pitts theory). We push $B$ towards $\Gamma_0$ with $P_t$ , $0\leq t\leq t'$ with $t'$ very close to $1$.
Let $C$ be a surface obtained by $A$ from a $\mu$-isotopy, close to $P_{t'}(B)$ as above in the varifold sense. 

Let $\{X_t\}_{t\in[0,1]}$ denote the surfaces constructed by isotopies as above, with $X_0=\Sigma_{b_q}$ and $X_1=C$. Similarly to Claims 1 and 2 in the proof of \cite[Deformation Theorem C]{MaNeindexbound}, one can choose $\mu$ and $\{X_t\}$ so that the following holds: there is a positive $\eta$ independent of the index $i$ of $\{\Sigma^i_t\}$ and $\mu$ when they are respectively large and small enough, such that 
$$
\forall t\in[0,1] \quad \mathbf{F}(|X_t|,V') > \eta
$$
for all stationary integral varifold $V'$ with mass $W(N,\Pi)$.

We can reapply $\gamma$-reduction to $P_{t'}(B)$ in each connected component of $\Omega_r$. 
We get a limit varifold $V_2$ of the form (\ref{form}). Let $B^*$ (resp. $D$) be a surface deformed from $P_{t'} (B)$ (resp. $C$) by a $\mu$-isotopy, arbitrarily close in the varifold sense to $V_2$. Now write 
$$V_2=m_{2,1}|S_1|+...m_{2,p} |S_p|.$$
Suppose that some of the coefficients are strictly larger than one, say $m_{2,j}>1$ for $j$ in a subset $\mathcal{J}$ of $\{1,...,p\}$. The surface $S_j$ is a $2$-sphere for any $j\in \mathcal{J}$. We are going to further deform $D$, by a $\mu$-isotopy, into a surface $E$ arbitrarily close in the varifold sense to 
\begin{equation} \label{frog}
V_3= m_{3,1}|S_1|+m_{3,2}|S_2|+...m_{3,p} |S_p|,
\end{equation}
where each $m_{3,i}$ is either $0$ or $1$. To begin this $\mu$-isotopy, we need to prepare the surface $D$. Note that by Remark \ref{aboutMSY}, $B^*$ is isotopic inside $\Omega_r$ to the union of $\tilde{S}_i$ and some small area spheres bounding balls in $\Omega_r$, where
$$\tilde{S}_i := \bigcup_{r=1}^{m_{2,i}} \{x\in N ; \text{dist}(x,S_i) = \bar{\delta} \frac{r}{m_{2,i}} \},$$
where $\bar{\delta}>0$ is arbitrarily small so that 
$$\Gamma_0\times [0,2\bar{\delta}) \to N$$
$$(x,s)\to \exp_x(s\nu)$$
is a diffeomorphism on its image supposed to be contained in $\Omega_{r/2}$, with $\nu$ the inward unit normal of $\Gamma_0$ (recall that $\Gamma_0 \subset \partial N$). The surface $D$ is obtained from $B^*$ with some thin handles and caps attached. By "thin handles", we mean that they are close to curves with endpoints in $B^*$ in the Hausdorff distance, say. 
Note that we could prove that $B^*$ is $\mu$-isotopic to $\bigcup_i \tilde{S}_i$ plus some small area spheres for $\mu$ small by adapting the proof of \cite{MSY} but for simplicity here we will only need that it is isotopic to the union of $\bigcup_i\tilde{S}_i$ with small area spheres. Let $\varphi$ be that isotopy.
        
By the local version of Remark 3.27 in \cite{MSY} shown in \cite[Section 7]{C&DL} with the blow down/blow up argument, there is a radius $r_0$ depending only on $\Gamma_0$ so that $B^*$ is $\mu$-isotopic to a surface which, up to some small area spheres bounding balls in $\Omega_r$, coincides with $B^*$ outside of a ball $b_{2r_0}$ of radius $2r_0$ centered at a point of $S_j$, and inside the concentric ball of half radius $b_{r_0}$, the surface is a made of a family of embedded disks $\{d_1,d_2,...,d_{l_0}\}$ ($l_0>1$) parallel to one another with area close to $\pi r_0^2$. Using $P_t$ we can also make these disks as close to each other as wanted. By attaching the thin handles, we get a resulting surface $\mu$-isotopic to $D$. We will still denote this prepared surface $D$. Our next goal is to find a small handle in $D$ and move it by isotopy to a vertical neck inside $b_{r_0}$ between to consecutive disks (which are components of $D\cap b_{r_0}$). This is because then it is easy to decrease the area of $D$ (without increasing it by more than $\mu$) by a uniform amount comparable to $2\pi r_0^2$ depending only on $\Gamma_0$. 

We attach the handles one by one to $B^*$. When attaching these handles, we stop at the first one that makes one of the layers
$$\varphi^{-1}(\{x\in N ; \text{dist}(x,S_i) = \bar{\delta} \})$$
and a layer
$$\varphi^{-1}(\{x\in N ; \text{dist}(x,S_{j_0}) = \bar{\delta} \frac{m_{2,1}-1}{m_{2,1}}\}) \quad \text{for some }j_0\in \mathcal{J}$$
into the same connected component; this handle exists because $D$ is connected. We can deform with a $\mu$-isotopy this handle into a handle joining two consecutive disks, say $d_1$ and $d_2$ (renumber the disks $d_l$ if necessary).
and close to a straight curve joining $d_1$ and $d_2$. 
This is possible due to the following fact. The surgery corresponding to this handle is on the side of $\Gamma_1$ and in a compression body $W$, any two curves linking $\partial_{+}W$ to a sphere component of $\partial_- W$ are isotopic inside $W$: indeed we can almost project each of these curves to $\partial_{+}W$ and by moving the roots of the curve, plus the fact the one can inverse the self-crossings by isotopy, we deform the curve to a chosen fixed one. 

After having deformed the critical handle to a vertical handle, we continue to add the remaining handles and get a surface $D'$ $\mu$-isotopic to $D$. Then by "opening up" the special handle, one decreases by a uniform amount (depending only on $\Gamma_0$) the area of ${D}'$ by a $\mu$-isotopy. This means that thanks to $\gamma$-reductions again in each connected of $\Omega_r$ we can deform by a $\mu$-isotpoy ${D}'$ into $D''$ arbitrarily close to $V'_2= m'_{2,1}|S_1|+...m_{2,p} |S_p|$, where $ m'_{2,j}<m_{2,j}$. We continue until we get $E$ from $D$ by a $\mu$-isotopy, such that $E$ is close to $V_3$ of the form (\ref{frog}). Since $\Sigma_{b_q}$ separates and $B(\Sigma_{b_q})$ has volume close to $0$, this means that $E$ is very close to $\Gamma_0$ in the varifold sense. By repeating the process described above using a sequence of smaller and smaller $\mu'$ going to zero and using the basin structure stated in the claim in the proof of Lemma \ref{canapply1}, we can join $E$ to $\Gamma_0$ (in the sense of Definition \ref{defsmoothsweep}) by a $\mu$-isotopy. To sum up what we just did: we deformed $\Sigma_{b_q}$ to $C$, then from $C$ we got a surface $E$ which we join to $\Gamma_0$, all by $\mu$-isotopies. 
Moreover,  if $\{Y_t\}_{t\in[0,1]}$ denote the surfaces constructed by isotopies as above from $C$ to $E$, arguing as in Claims 3 and 4 in the proof of \cite[Deformation Theorem C]{MaNeindexbound}, one can choose $\mu$ and $\{Y_t\}$ so that for a positive $\eta$ independent of the index $i$ of $\{\Sigma^i_t\}$ and $\mu$ when they are respectively large and small enough,  
$$
\forall t\in[0,1],\quad\mathbf{F}(|Y_t|,V') > \eta
$$ 
for all stationary integral varifold $V'$ with mass $W(N,\Pi)$.
By reversing the directions of the isotopies, we get an isotopy joining $\Gamma_0$ to $\Sigma_{b_q}$, and we glue this isotopy to $\{\Sigma_t\}_{t\in[b_{b_q},1]}$.

The above procedure can be realized in a symmetric way for $t=c_1$ (remember that $[c_1,d_1]$,...,$[c_{q'},d_{q'}]$ are the intervals composing $\mathbf{V}_{\alpha} \cap (b_q,1]$). As a result, we get a new family $\{\hat{\Sigma}_t\}_{t\in[b',a']}$, where $0\leq b'<a'\leq 1$. 
After reparametrizing, we obtain a family still denoted by $\{\hat{\Sigma}_t\}_{t\in[0,1]}$, satisfying the following properties for some $\eta$ small enough and for $\mu$ arbitrarily small:
\begin{itemize}
\item $\{\hat{\Sigma}_t\} \in \Pi$,
\item $\max_t Area(\hat{\Sigma}_t) \leq \max_t Area({\Sigma}_t) + \mu$,
\item $\mathbf{F}(|\hat{\Sigma}_t|,V) > \eta$ for all $t\in[0,1]$ and all $V$ of the form (\ref{form}) with $\mathbf{M}(V)=W(N,\Pi)$,
\item $\mathbf{F}(|\hat{\Sigma}_t|,V') > \eta$ for all $t\in[0,1]$ and all stationary integral varifold $V'$ with $\mathbf{M}(V')=W(N,\Pi)$ which are not in $\mathbf{\Lambda}(\{\{{\Sigma}^i_t\}\}_i)$.
\end{itemize}
We emphasize that this sweepout $\{\hat{\Sigma}_t\}$ is not a priori homotopic to the original sweepout $\{\Sigma\}_t$.

Now we can conclude the proof. Let $\{\{\Sigma^i_t\}\}_i$ be a pulled-tight minimizing sequence and $\mu_i$ a sequence going to $0$. Transforming each sweepout of this sequence as above with parameter $\mu_i$ instead of $\mu$ (but keeping $\alpha>0$ fixed) produces a new pulled-tight minimizing sequence $\{\{\hat{\Sigma}^i_t\}\}_i\subset \Pi$. Note similarly to \cite[Deformation Theorem C]{MaNeindexbound} that by construction:
\begin{align} \label{inclusi}
\begin{split}
 & \mathbf{\Lambda}(\{\{\hat{\Sigma}^i_t\}\}_i) 
\cap \{\text{stationary integral varifolds of mass $W(N,\Pi)$}\} \\
\subset  \quad&  \mathbf{\Lambda}(\{\{{\Sigma}^i_t\}\}_i) 
\cap \{\text{stationary integral varifolds of mass $W(N,\Pi)$}\}.
\end{split}
\end{align} 
But by construction of $\{\{\hat{\Sigma}^i_t\}\}_i$, any varifold $V' \in \mathbf{\Lambda}(\{\{\hat{\Sigma}^i_t\}\}_i)$ with $||V'||(N)=W(N,\Pi)$ is $\eta$-far from any varifold of the form (\ref{form}). So the usual min-max theorem (see \cite{C&DL}) would produce a varifold in the first intersection in (\ref{inclusi}) with smooth support, i.e. a minimal surface whose area counted with multiplicity is $W(N,\Pi)$, and which is not entirely contained in the boundary $\partial N$. This contradicts our assumption on $\mathbf{\Lambda}(\{\{{\Sigma}^i_t\}\}_i)$ so the theorem is proved.

\end{proof}

It is expected more generally that, similarly to the Almgren-Pitts setting (see \cite{MaNeindexbound}), oriented local min-max surfaces cannot be stable for bumpy metrics.

\begin{remarque} \label{topolo}
Strong irreducibility has consequences on how surgeries can be performed. The following observation was used in previous works but not really explained. Suppose that $S=S_1$ is strongly irreducible in $N$ and separates $N$ into $W$, $W'$. Suppose also that $S_2$, ..., $S_K$ are successively obtained by surgery from the previous one (for instance along a $\gamma$-reduction \cite{MSY} or a smooth min-max procedure \cite{Ketgenusbound}). We can assume the surgery curves $\alpha_i$ ($i\in\{1,...,K\}$) to be disjoint and contained in $S_1$. Then by switching $W$ and $W'$ if necessary, for all $i \in \{1,...,K\}$, if the surgery is performed along a curve $\alpha_i$ essential in $S_i$, $\alpha_i$ bounds a disk in $W$. This follows from \cite[Theorem 2.1]{Scharlemann}.


From this "one-sided surgeries" property, we deduce two useful facts \cite{Rubinsteinnotes}, \cite[Theorem 3.3]{KeMaNe}. Firstly if the surfaces of a smooth sweepout are isotopic to a strongly irreducible Heegaard splitting $H$ as in Theorem \ref{smoothminmax} and a min-max surface is non-orientable, then there is only one non-sphere component, which has multiplicity two when not an $\mathbb{R}P^2$, and its double cover plus a vertical handle is isotopic to $H$ (see \cite[Corollary 1.6]{Heath}). Secondly if the min-max surface is orientable, then the non-sphere components have multiplicity one (this can be checked using \cite[Theorem 2.11]{ScharlemannThompson}).
\end{remarque}

\begin{remarque}\label{aboutMSY}
For a $j\in\{1,...,p\}$, let $B'$ be the intersection of $B^*$ and the $r$-neighborhood of $S_{2,j}$. By strong irreducibility of the Heegaard splitting $H$, it suffices to prove that the number of components of $B'$ projecting with degree $\pm1$ on $S_{2,j}$ with $P_1$ is exactly $m_{2,j}$. Indeed the connected components project each with degree either $\pm1$ or $0$ by embeddedness and orientability. The components of degree $0$ bound a handlebody inside the $r$-neighborhood so they are spheres bounding balls included in the $r$-neighborhood. These spheres have small area since otherwise one could use $P_t$ and reduce their area by a uniform amount with a $\mu$-isotopy inside the $r$-neighborhood, which would contradict the fact that $V_2$ (which is close to $B^*$) was obtained by $\gamma$-reduction from $P_t(B)$. The number of components with degree $\pm1$ cannot exceed $m_{2,j}$ for $B^*$ has area close to the mass of $V_2$. Now if the number of components with degree $\pm1$ is less than $m_{2,j}$ then the area of one of these components is larger than say $(1+\frac{1}{m_{2,j}})\mathcal{H}^2(S_{2,j}) $ ($B^*$ being very close to $V_2$ as varifolds). But using $P_t$ ($t$ close to $1$) one can decrease the area of such a component by a uniform amount without increasing it (one may move the other components as well during the process). That contradicts the fact that $V_2$ is the limit obtained by $\gamma$-reduction with $\mu$-isotopies. Thus the number of components with degree $\pm 1$ is exactly $m_{2,j}$, and by strong irreducibility there are isotopic together to graphs over $S_{2,j}$.

\end{remarque}

The following lemma is the smooth analogue of Lemma \ref{canapply1}.

\begin{lemme} \label{canapply2}
Assume $N \subset (M,g)$, $\Gamma_0$, $\Gamma_1$ to be as in Theorem \ref{smoothminmax}. Let $\Pi$ be a saturated set whose elements satisfy: for any smooth sweepout $\{{\Sigma}_t\}$ in $\Pi$, $\mathcal{A}(\{{\Sigma}_t\}) = [|N|]$, and $\Sigma_i=\Gamma_i $ for $i=0,1$. Suppose that ${\Gamma}_0$ is a stable minimal surface. Then
$$ W(N,\Pi) > \mathcal{H}^2({\Gamma}_0).$$

Consequently if ${\Gamma}_0$ and ${\Gamma}_1$ are both stable and $\Pi$ satisfies the hypotheses of Theorem \ref{smoothminmax}, then its conclusion holds. 

\end{lemme}

\begin{proof}
Notice the following general inequality relating the Simon-Smith and Almgren-Pitts widths. By discretizing a sequence of minimizing smooth sweepouts $\{\Sigma^i_t\}\in \Pi$, we get a homotopy sequence of mappings $\{\psi_i\} \in \Pi \in \pi_1^\sharp(\mathcal{Z}_2(N),C_0,C_1)$ where the currents $C_0$, $C_1$ are determined by $\Sigma_0$, $\Sigma_1$. As a consequence, 
$$W(N,\Pi) \geq \mathbf{L}(\Pi).$$
The proof then follows immediately from Lemma \ref{canapply1} 
\end{proof}

\section{Minimal Heegaard splittings in orientable 3-manifolds} \label{Rubinstein conj}

In this section, $M$ is an irreducible oriented closed $3$-manifold not diffeomorphic to $S^3$. All surfaces considered are closed embedded. For simplicity, we will use "minimal Heegaard splitting" to denote a closed connected embedded minimal surface which is a Heegaard splitting.

\subsection{}\textbf{Main theorem}

We now state the main theorem.

\begin{theo} \label{positivegenus}
Let $(M,g)$ be a closed oriented $3$-manifold not diffeomorphic to the $3$-sphere. Suppose that there is a strongly irreducible Heegaard splitting $H$. Then either $H$ is isotopic to a minimal surface of index at most one, or isotopic to the stable minimal oriented double cover of a non-orientable minimal surface with a vertical handle attached.

\end{theo}


Of course for $3$-spheres, the theorem of Simon-Smith gives the existence of a minimal $2$-sphere of index at most one \cite{Smith} \cite{C&DL}. The existence of a strongly irreducible Heegaard splitting forces $M$ to be irreducible.

Notice the following. If $h$ is the Heegaard genus of $M$ and if $M$ contains an incompressible surface of positive genus $k$ less than $h$, then by \cite{FreedmanHassScott} 
\begin{itemize}
\item there is an oriented area-minimizing surface of genus $k$,
\item or there is a non-orientable minimal surface whose oriented double cover is area minimizing and of genus $k$.
\end{itemize}
On the other hand, if $M$ is irreducible and does not contain any such surfaces (for instance when $M$ is non-Haken), then there is a strongly irreducible Heegaard splitting so Theorem \ref{positivegenus} applies. This follows from \cite[Theorem 3.1]{CassonGordon}, whose proof implies the following: if $H$ is a Heegaard splitting of genus $h'$ which is not strongly irreducible, then either it is reducible or $M$ contains an incompressible surface of positive genus less than $h'$.

We list a few corollaries of Theorem \ref{positivegenus}.

\begin{coro}
The conjecture of Rubinstein is true.
\end{coro}

\begin{coro}
Any lens space not diffeomorphic to $S^3$ or $\mathbb{R}P^3$ contains a minimal Heegaard splitting of genus one with index at most one. 
\end{coro}

\begin{coro}
Any $(\mathbb{R}P^3,g)$ contains
\begin{itemize}
\item either a Heegaard splitting of genus one and index at most one, 
\item or a minimal $\mathbb{R}P^2$ with stable oriented double cover.
\end{itemize}

If the metric $g$ is bumpy, then 
\begin{itemize}
\item either there is an index one Heegaard splitting of genus one,
\item or there is an index one minimal sphere.
\end{itemize}
\end{coro} 

\begin{proof}
For the bumpy metric case, if there is a minimal $\mathbb{R}P^2$ with stable oriented double cover, then we get a $3$-ball by cutting $\mathbb{R}P^3$ along the previous projective plane. As in Corollary \ref{existence core}, we check that there is a core whose interior does not contain the boundary of the $3$-ball and applying Theorem \ref{smoothminmax} to this core, we get an interior index one minimal sphere.
\end{proof}

\textbf{Example:}
This corollary is optimal if we focus on index at most one tori not included in a $3$-ball, as shown by the following example. Consider a long cylindrical piece $[0,1]\times S^2$, cap it on one side (say $\{0\}\times S^2$) with a half-sphere, then take the quotient on the other side ($\{1\}\times S^2$) by the usual $\mathbb{Z}_2$-action to get an $\mathbb{R}P^3$ with positive scalar curvature. If the cylindrical piece is long enough with a metric near the product metric, and if the spheres $\{t\}\times S^2$ constitute a mean convex foliation with mean curvature vector pointing towards $\{1\}\times S^2$, then by the maximum principle, the monotonicity formula and the area bound \cite[Proposition A.1 (i)]{MaNe}, there is no index one minimal torus intersecting the projective plane $(\{1\}\times S^2) / \mathbb{Z}_2$.

Before proving Theorem \ref{positivegenus}, we need two lemmas. Recall that $(M,g)$ is  closed irreducible oriented and not diffeomorphic to the $3$-sphere.

\begin{lemme} \label{approximation}

Let $(M,g)$ be as above, let $H$ be an irreducible Heegaard splitting. Suppose that $H$ is not isotopic to a stable minimal surface or to the stable oriented double cover of a non-orientable minimal surface with a vertical handle attached. Let $\mathfrak{C}$ be an $H$-core (which exists by Corollary \ref{existence core}). Then for all constant $C>0$ large enough, there is a sequence of bumpy metrics $g_m $ converging smoothly to $g$ and $C$-bounded $H$-cores $ \mathfrak{C}_{m}$ with respect to $g_m$, such that $\partial \mathfrak{C}_{m}$ converges smoothly to $\partial \mathfrak{C}$ (with respect to $g$).
\end{lemme}

\begin{proof}
We choose a function $\tilde{\lambda}_m$ converging smoothly to $1$ so that $\partial \mathfrak{C}$ is strictly minimal for $\tilde{\lambda}_m g$. A small neighborhood $V_m$ of $\partial \mathfrak{C}$ then has strictly mean convex boundary $\partial V_m$ for $\tilde{\lambda}_m g$, and we can make $\partial V_m$ be diffeomorphic to two copies of $\partial \mathfrak{C}$, each one converging to $\partial \mathfrak{C}$ on one side as $m\to \infty$. Then using the genericity of bumpy metrics proved in \cite{Whitebumpy}, we modify slightly the metric $\tilde{\lambda}_m g$ into $g_m$ so that the stable surfaces of $(M,g_m)$ are strictly stable and $\partial V_m$ still has a mean convex boundary. By minimizing half of the boundary $\partial V_m$ inside $V_m$, one finds a strictly stable embedded minimal surface $S_m$ for $g_m$ so that $S_m$ converges to $\partial \mathfrak{C}$. Let $C$ be any constant larger than twice the area of $\partial \mathfrak{C} \subset (M,g)$. For $m$ large, Corollary \ref{existence core} gives the existence of a $C$-bounded $H$-core $\mathfrak{C}_m$ with
$$S_m \subset M\backslash \interior(\mathfrak{C}_m).$$
As $m$ tends to infinity, the boundary $\partial \mathfrak{C}_m$ converges smoothly by \cite{Sc} and the limit is $\partial \mathfrak{C}$, by definition of the $H$-core $\mathfrak{C}$.

\end{proof}

\begin{lemme} \label{constructsweep}  
Suppose that $(N,g)$ is a compression body endowed with $g$ a bumpy metric, such that $\partial_+N = \Gamma$ is mean-convex and $\partial_- N= \Gamma'$ is a disjoint union of stable minimal surfaces.
Then one of the following cases occurs:
\begin{itemize}

\item either there is a compression body $\tilde{N}$
such that the hypotheses above are true with $\Gamma'$ (resp. $N$) replaced by $\tilde{\Gamma}':=\partial_-\tilde{N}$ (resp. $\tilde{N}$), and moreover $\tilde{N}$ is strictly included in $N$,
\item or there is a smooth sweepout $\{\hat{\Sigma}_t\}_{t\in[0,1]}$ of $N$ with $\mathcal{A}(\{\hat{\Sigma}_t\}) = [|N|]$, $\hat{\Sigma}_0 = \Gamma$, $\hat{\Sigma}_1=\Gamma'$, $\hat{\Sigma}_t$ is isotopic to $\Gamma$ for $t\in[0,1)$, and moreover
$$\max_{t\in[0,1]} \mathcal{H}^2(\hat{\Sigma}_t) =\mathcal{H}^2(\Gamma).$$
\end{itemize}

\end{lemme}

\begin{proof}

Let us minimize the area of a surface close to $\Gamma$ (and with smaller area) inside $N$ using the $\gamma$-reduction ``with constraints" (see \cite[Section 7]{C&DL}). We get a limit stable minimal surface $S$. By the isotopic deformation arguments and the same topological arguments used in the proof of Theorem \ref{smoothminmax} (see Remark \ref{topolo}), we get that either $S$ is different from $\Gamma'$ in which case the first item in the lemma occurs, or $S$ is $\Gamma'$ and the second item occurs.


\end{proof}

\begin{proof}[Proof of Theorem \ref{positivegenus}]
The surfaces considered are closed and embedded. Let $H$ be a strongly irreducible Heegaard splitting of $M$. First suppose that the metric is bumpy. We can assume that $H$ is not isotopic to a stable minimal surface or to the stable oriented double cover of a non-orientable minimal surface $\Sigma \subset U$ with a vertical handle attached. So in particular there is an $H$-core $\mathfrak{C}$ by Corollary \ref{existence core}. Let $\Gamma_0 \cup \Gamma_1 =\partial \mathfrak{C} $ be a boundary decomposition with $\Gamma_i = \partial_-W_i'$, where $(W_0',W_1',H')$ is an $H$-compatible splitting. Let $\Pi$ be the saturated set generated by all the smooth sweepouts $\{\Sigma_t\}$ of $\mathfrak{C}$, such that $\mathcal{A}(\{\Sigma_t\}) = [|\mathfrak{C}|]$, such that $\Sigma_0$ (resp. $\Sigma_1$) is $\Gamma_0$ (resp. $\Gamma_1$) such that the slices $\Sigma_t$ ($0<t<1$) are smooth connected orientable and are isotopic to $H$. We can apply Theorem \ref{smoothminmax} to the $H$-core $\mathfrak{C}$ and we get an embedded minimal surface $\Sigma$ with area $W(\mathfrak{C},\Pi)$ (taking into account multiplicities).

If $\Sigma$ is oriented, it is obtained by surgeries of $H$. By strong irreducibility (see Remark \ref{topolo}) every time a surgery occurs along an essential curve, the surgery disk is on one well-defined side of $H$ independent of the surgery. Surgeries along non-essential curves can happen on both sides and split off spheres. Notice that we can suppose there is no minimal sphere in the interior of the $H$-core. Otherwise we can minimize its area on the non-trivial side: either we get an $\mathbb{R}P^2$ component with stable universal cover and so $M$ is actually an $\mathbb{R}P^3$, or we get some stable spheres in $\interior(\mathfrak{C})$ and we could have added these spheres to $\Gamma_0$ and get a smaller $H$-compatible couple, contradicting the definition of the $H$-core $\mathfrak{C}$ (here we are using that $M$ is not a $3$-sphere). The components of $\Sigma$ bound compression bodies or are included in the boundary of $\mathfrak{C}$, and none of the latter components which is not a sphere is included in one of the compression bodies previously mentioned. Since $H$ is strongly irreducible, the multiplicity of the non-sphere components is one (see Remark \ref{topolo}). By Theorem \ref{smoothminmax}, one of the component $\Sigma'$ is contained in the interior $\interior(\mathfrak{C})$, has positive genus, has multiplicity one, and is unstable by definition of an $H$-core again. There is exactly one such unstable component $\Sigma'$ by \cite{MaNeindexbound}. Suppose $H$ is not isotopic to $\Sigma'$. We minimize its area on the side $M'\subset \interior(\mathfrak{C}) $ which is not a compression body (say the side of $\Gamma_1$). One checks that $\Sigma'$ is incompressible inside $M'$ as follows: suppose that an essential curve $\tilde{\gamma}$ on $\Sigma'$ bounds a disk on one side, it has to be the side of $M'$, but since $H$ is not isotopic to $\Sigma'$ but isotopic to $\Sigma'$ with necks linking to some other components and some handles added to them (the necks and handles correspond to surgeries), there is an essential curve on one of the small necks or handles, disjoint from the disk bounded by $\tilde{\gamma}$ and bounding a disk on the other side of $H$, contradicting the strong irreducibility. So $\Sigma'$ is incompressible inside $M'$ and by $\gamma$-reduction one gets a stable minimal surface $\Sigma''\subset M'$. Either a component of $\Sigma''$ is non-orientable, then its oriented double cover is stable and by adding a vertical handle we get a surface isotopic to $H$ (by strong irreducibility of $H$): this case cannot happen by our assumption in the beginning of the proof. Or $\Sigma''$ is oriented: then replacing $\Gamma_0$ by $(\Gamma_0 \cup\Sigma'')\backslash \Sigma'$ (and removing the spheres inside compression bodies if necessary), we get a contradiction for the definition of the $H$-core $\mathfrak{C}$. So it means that the surface $\Sigma$ has one component which is isotopic to $H$.

If $\Sigma$ is non-orientable, then by the topological claim in the proof of \cite[Theorem 3.3]{KeMaNe} and \cite{Ketgenusbound} or Remark \ref{topolo}, there is a unique component $\Sigma'$ of $\Sigma$ in the interior $\interior(\mathfrak{C})$. It is non-orientable, has multiplicity at least two and for $\epsilon>0$ small enough, each $\Gamma_t$ is isotopic to $\partial N_\epsilon (\Sigma')$ with a vertical handle attached. 
Remember that the oriented double cover $\tilde{\Sigma'}$ of $\Sigma'$ is unstable by our assumption in the beginning of the proof. We now apply Lemma \ref{constructsweep} to the complementary inside the $H$-core $\mathfrak{C}$ of a small neighborhood of $\Sigma'$. By definition of an $H$-core, only the second item of this lemma is valid. Hence as in \cite[Theorem 3.3]{KeMaNe}, the unstability of $\tilde{\Sigma'}$ gives rise to a smooth sweepout $\{\tilde{\Gamma}_t\} \in \Pi$, such that
$$W(\mathfrak{C}, \Pi) \leq \max_{t\in[0,1]} \mathcal{H}^2(\Gamma_t) < 2\mathcal{H}^2(\Sigma').$$
Since $2\mathcal{H}(\Sigma')\leq W(\mathfrak{C}, \Pi) $, it is a contradiction. So this non-orientable case cannot happen, and the first part of the theorem is proved in the case of bumpy metrics.

Finally if the metric is not bumpy, we use Lemma \ref{approximation} to approximate $g$ by bumpy metrics $g_m$. For each $m$, there is a $C$-bounded $H$-core $\mathfrak{C}_m$ and $W(\mathfrak{C}_m, \Pi)$ converges to $W(\mathfrak{C}, \Pi)$. If $C$ is chosen bigger than $2W(\mathfrak{C}, \Pi)$, then we check without difficulty that the above arguments hold for a $C$-bounded $H$-core instead of an $H$-core for large $m$, since $\lim_m W(\mathfrak{C}_m,\Pi) = W(\mathfrak{C},\Pi)$. We get for each $m$ large enough 
\begin{itemize}
\item either a minimal surface of at most index one isotopic to $H$,
\item or the stable minimal oriented double cover of a non-orientable minimal surface such that when we attach a vertical handle, we get a surface isotopic to $H$.
\end{itemize}
These minimal surfaces have area at most $C$. Subsequently this sequence converges by \cite{Sharp} to a minimal surface $\Sigma^*$. By strong irreducibility, either the limit is two-sided and the convergence is smooth, or the limit is one-sided, the oriented double cover of $\Sigma^*$ is stable (by a Jacobi field argument, see \cite{Sharp} for instance) and $H$ is isotopic to it with a vertical handle attached. So $\Sigma^*$ is as in the statement of Theorem \ref{positivegenus}.


\end{proof}

\subsection{Case of non-prime oriented 3-manifolds} \label{general case}

When an oriented $3$-manifold is not prime, we can cut it along area-minimizing $2$-spheres obtained by minimizing the area of the separating essential spheres. They are either embedded or the double-cover of a projective plane.  Let $C_0$ be one of the component and denote by $\hat{C}_0$ the manifold obtained by closing with $3$-handles (i.e. gluying $3$-balls). If $\hat{C}_0$ has a strongly irreducible Heegaard splitting $H$, which we can suppose included in $C_0$, then Theorem \ref{positivegenus} applies and we obtain:

\begin{theo} \label{nonprime1}
Let $(\bar{M},g)$ be an oriented $3$-manifold and $\hat{C}_0$ as above with a strongly irreducible Heegaard splitting $H\subset C_0$. Then either $H$ is isotopic to a minimal surface of index at most one, or isotopic to the stable minimal oriented double cover of a non-orientable minimal surface with a vertical handle attached.

\end{theo}

\section{Minimal Heegaard splittings in orientable 3-manifolds with positive scalar curvature} \label{R>0}

In this section, we specialize to the case of positive scalar curvature. We suppose that $M$ is a spherical space form $S^3/\Gamma$ not diffeomorphic to $S^3$. Assume that $M$ is endowed with a metric $g$ with scalar curvature at least $6$, then any embedded orientable stable surface is a sphere (see \cite{SchoenYau}) of area at most $4\pi/3$, and index one minimal surfaces also have an area bound \cite[Proposition A.1 (i)]{MaNe}.

\subsection{}\textbf{Existence of index one minimal Heegaard splitting with minimal genus}


A corollary of Theorem \ref{positivegenus} is the following existence theorem for irreducible minimal Heegaard splittings in spherical space forms with positive scalar curvature.
\begin{theo} \label{boomchak}
\begin{enumerate}
\item
Let $(M,g)$ be an $\mathbb{R}P^3$ which has scalar curvature at least $6$. Then 
\begin{itemize}
\item either $M$ contains an index one Heegaard splitting of genus one which has area less than $4\pi$,
\item or $M$ contains a minimal $\mathbb{R}P^2$ with stable oriented double cover and area less than $2\pi$.
\end{itemize}

\item
Let $(M,g)$ be a spherical space form $S^3/\Gamma$ not diffeomorphic to $S^3$ or $\mathbb{R}P^3$, and with positive scalar curvature. Let $ H$ be a strongly irreducible Heegaard splitting of $M$. Then $M$ contains an index one minimal surface isotopic to $H$.

\end{enumerate}
\end{theo}

Note that, since a spherical space form admits a positive curvature metric, it does not contain incompressible surfaces so any Heegaard splitting of minimal genus is strongly irreducible by \cite{CassonGordon}.

\begin{proof}
Observe that since the scalar curvature is positive any oriented stable minimal surface is a sphere, and a non-orientable minimal surface with stable oriented double cover has to be an $\mathbb{R}P^2$. Then the theorem follows essentially from Theorem \ref{positivegenus}. It remains to get the strict area upper bound in the case where $M$ is diffeomorphic to $\mathbb{R}P^3$. First by \cite{BrayBrendleEichmairNeves}, the area of an area-minimizing $\mathbb{R}P^2$ is smaller than $2\pi$ whenever $M$ is not round. The theorem is clearly true for the round $\mathbb{R}P^3$ (consider the projection of a Clifford torus). So let $\Sigma$ be such an area-minimizing projective plane and suppose that its oriented double cover is unstable. Attaching a vertical handle to the boundary of a thin tubular neighborhood of $\Sigma$, we get a genus one Heegaard splitting $H$ and there is an $H$-core $\mathfrak{C}$ containing $\Sigma$ in its interior (by the arguments of Lemma \ref{existence core}). As in the proof of Theorem \ref{positivegenus}, we use Lemma \ref{constructsweep} to find a smooth sweepout of $\mathfrak{C}$ with maximum area less than $2\mathcal{H}^2(\Sigma)$, so the local min-max procedure gives either an index one Heegaard splitting of genus one and area less than $4\pi$, or a minimal projective plane with area less than $2\pi$ and stable universal cover.

\end{proof}

\subsection{} \textbf{Existence of index one minimal Heegaard splitting with area less than $4\pi$}

The following result combined with Theorem \ref{boomchak} (1) can be thought of as an extension of Theorem 23 in \cite{Antoine} in the orientable case.

\begin{theo} \label{boomchak2}

Let $(M,g)$ be a spherical space form $S^3/\Gamma$ not diffeomorphic to $S^3$ or $\mathbb{R}P^3$, and with scalar curvature at least $6$. Then $M$ contains an index one minimal Heegaard splitting with area less than $4\pi$.

\end{theo}

\begin{proof}

Let $M$ be a spherical space form not diffeomorphic to $S^3$, with scalar curvature at least $6$. Recall that if $M$ cannot contain an embedded $\mathbb{R}P^2$ since then by irreducibility $M$ would be an $\mathbb{R}P^3$. By Corollary \ref{existence core}, there is a core $\mathfrak{C}\subset M$. Observe that the boundary components are all spheres.

We now work in the discrete setting of Almgren-Pitts (see Section \ref{prelim}). Pick a homotopy class $\Pi_\mathfrak{C} \in \bigcup_{C_0,C_1} \pi^{\sharp}_1(\mathcal{Z}_{n}(M),C_0,C_1)$ in the sense of Almgren-Pitts corresponding to the fundamental class of $\mathfrak{C}$, with $C_0$, $C_1$ going through the (finitely many) possible choices, namely $\mathcal{A}(\Pi_\mathfrak{C}) = [|\mathfrak{C}|]$. The width of the core is $\mathbf{L}(\Pi_\mathfrak{C})$. Suppose the latter to be minimal with respect to the choice of $C_0$, $C_1$. 

\textbf{Claim:} There is an embedded connected minimal surface of index 1, which is also a Heegaard splitting of $M$, such that
$$\Sigma\subset \interior(\mathfrak{C}) \text{ and } \mathcal{H}^2(\Sigma) = \mathbf{L}(\Pi_\mathfrak{C}).$$
Indeed for a bumpy metric, we get a minimal surface $\Sigma$ in the interior of the core by Theorem \ref{discreteminmax}. Suppose that $\Sigma$ is oriented unstable. It has to be a Heegaard splitting because minimizing its area in the core on one side, we get stable spheres (which have to bound $3$-balls in this side) so we can make its area go to zero in this side (see \cite[Proposition 1]{MSY}). Thus applying Lemma \ref{constructsweep} on both sides of $\Sigma$, using interpolation and discretizing the sweepouts, we get that actually $\mathcal{H}^2(\Sigma) = \mathbf{L}(\Pi_\mathfrak{C})$. As in Theorems \ref{positivegenus} and \ref{boomchak} we rule out the case where $\Sigma$ is non-orientable. For general metrics, we use the approximation proved in Lemma \ref{approximation}. 

It remains to improve the area bound yielded by the usual Hersch trick. Suppose for simplicity that the metric is bumpy. We want to show $\mathbf{L}(\Pi_\mathfrak{C}) < 4\pi.$ For that purpose, let us consider the lift $\tilde{\mathfrak{C}}$ of $\mathfrak{C}$ to $S^3$. We apply Theorem \ref{smoothminmax}. There is a minimal sphere $\tilde{S}$ in the interior of $ \tilde{\mathfrak{C}}$ and its index is at most one so by the Hersch trick, $$\mathcal{H}^2(\tilde{S})\leq 4\pi.$$
After projecting on $M$, we get a non-embedded immersed minimal surface $S\subset \interior(\mathfrak{C})\subset M$. We want to apply the method of \cite[Proposition 19, Proposition 22]{Antoine} to $S$. If at embedded points of $S$ the projection is a local diffeomorphism then any generic closed curve in $\interior(\mathfrak{C})$ intersecting $S$ an odd number of times would lift to a closed curve intersecting $\tilde{S} \subset S^3$ an odd number of times, which is absurd. Hence in that case, inspecting the proof of \cite[Lemma 16]{Antoine}, we conclude that $S$ satisfies the local separation property (\textbf{LS}). If at embedded points of $S$ the projection is a double cover, then $2 \mathcal{H}^2(S) \leq 4\pi$ and we can simply reduce the boundary of every component of $\mathfrak{C} \backslash S$ at the same time. So in every case we construct a homotopy sequence $\{\psi_i\}\in \Pi_\mathfrak{C}$ such that 
$$\mathbf{L}(\{\psi_i\}) \leq 4\pi
 \text{ and } \mathbf{L}(\Pi_\mathfrak{C}) < \mathbf{L}(\{\psi_i\}),$$
the strict inequality coming from the fact that $S$ is not embedded while Almgren-Pitts min-max theory produces embedded surfaces (see \cite{Antoine}).

Finally, if the metric is not bumpy, we use Lemma \ref{approximation} to get minimal spheres $\tilde{S}_m$ in the interior of the lifts $\tilde{\mathfrak{C}}_m$. Since ${\mathfrak{C}}_m$ are cores, $\tilde{S}_m$ cannot converge to a sphere of $\partial \tilde{\mathfrak{C}}$ so there is a limit minimal sphere in the interior of $ \tilde{\mathfrak{C}}$, which projects to $S$ in $M$. The end of the argument is the same as previously. 

\end{proof}

\begin{remarque} \label{physalis}
\begin{itemize}
\item Here is an observation following from \cite[Proposition 1]{MSY}: if $M$ is a spherical space form with Heegaard genus two and positive scalar curvature, then any minimal torus is included in a $3$-ball.
\item The above theorem does not provide information on the genus of the Heegaard splitting with area less than $4\pi$ since we are using a projection argument and the min-max theory of Almgren-Pitts. For $M$ as in the theorem, is there a Heegaard splitting with area less than $4\pi$, which has minimal genus? In the special case $Ric>0$, this is true by combining \cite{MaNe} and \cite{KeMaNe} but it seems quite arduous to extend the method to the case $R>0$ (see \cite{Antoinespheres}).
\end{itemize}
\end{remarque}

\subsection{Case of non-prime closed oriented 3-manifolds with positive scalar curvature}

An oriented closed manifold with positive scalar curvature is diffeomorphic to the connected sum of spherical space forms and some $S^1\times S^2$'s \cite{SchoenYau, GromovLawson, Perelman1, Perelman2, Perelman3} (see also \cite[Corollary 0.5, Chapter 15]{MorganTian}, \cite[Theorem 7.1, b)]{Marquesscalar}). When $M$ is not a $3$-sphere, write 
\begin{align} \label{decomposition prime}
\begin{split}
M\approx \quad&  a.\mathbb{R}P^3 \# S^3/\Gamma_1\#...\# S^3/\Gamma_b\\ 
& \# S^3/\tilde{\Gamma}_1\#...\# S^3/\tilde{\Gamma}_c \# d. S^1\times S^2
\end{split}
\end{align}
where the factors $S^3/\Gamma_i$ are not diffeomorphic to $\mathbb{R}P^3$ and do not contain embedded non-orientable surfaces, the factors $S^3/\tilde{\Gamma}_i$ are not diffeomorphic to $\mathbb{R}P^3$ and contain an embedded non-orientable surface. Here $a$ (resp. $d$) is the number of $\mathbb{R}P^3$ (resp. $S^1\times S^2$) in the prime decomposition.

Recall that $S^1\times S^2$ is the only orientable non-irreducible prime $3$-manifold. As in subsection \ref{general case}, everything extends to the non prime case.

\begin{theo} \label{nonprime2}
Let $(\bar{M},g)$ be an oriented $3$-manifold not diffeomorphic to $S^3$, with scalar curvature at least $6$, and let $\hat{C}_0$ be as in Theorem \ref{nonprime1}. Suppose that $H \subset C_0$ is a strongly irreducible Heegaard splitting of $\hat{C}_0$. 

\begin{itemize}
\item
When $ \hat{C}_0$ is not an $\mathbb{R}P^3$, then $H$ is isotopic to a minimal surface of index one, and there is an index one minimal surface $\Sigma\subset C_0$ of area less than $4\pi$, such that $\Sigma$ is a Heegaard splitting in $\hat{C}_0$.
\item
When $\hat{C}_0$ is an $\mathbb{R}P^3$, either there is an index one minimal torus isotopic to $H$ of area less than $4\pi$ or a minimal $\mathbb{R}P^2$ of area less than $2\pi$ with stable universal cover $\tilde{S}$, and $H$ is isotopic to $\tilde{S}$ plus a vertical handle attached.
\end{itemize}
\end{theo} 

Let $\mathcal{I}(I)$ be the family of minimal surfaces with index at most $I$. When the scalar curvature is positive, Theorem 1.5 in \cite{Carlottofinite} shows that generically $\mathcal{I}(I)$ is finite for all integer $I$. A corollary of our previous theorem is a lower bound for the cardinal of $\mathcal{I}(1)$ for any metric:

\begin{coro}
Let $(\bar{M},g)$ be as in (\ref{decomposition prime}) and with positive scalar curvature. Then $\mathcal{I}(1)$ has at least $a+2b+3c+2d-1$ elements.  

\end{coro}

\bibliographystyle{plain}
\bibliography{biblio18_09_25}

\end{document}